\pdfoutput=1
\documentclass[11pt]{article}

\usepackage[margin=1in]{geometry}
\usepackage[T1]{fontenc}
\usepackage[utf8]{inputenc}
\usepackage{times}

\usepackage{amsmath,amssymb,amsfonts,amsopn,mathtools}
\usepackage{amsthm}
\usepackage{bm}
\usepackage{dsfont}
\usepackage[nocomma]{optidef}
\usepackage{commath}

\usepackage{graphicx}
\usepackage{epsfig}
\usepackage{epstopdf}
\usepackage{float}
\usepackage{subcaption}
\usepackage{stfloats}

\usepackage{algorithm}
\usepackage{algorithmic}

\usepackage{enumitem}
\usepackage{xcolor}
\usepackage{cite}
\usepackage{xurl}
\usepackage[hidelinks]{hyperref}
\usepackage[capitalise,nameinlink]{cleveref}

\DeclareGraphicsExtensions{.pdf,.png,.jpg,.eps}

\makeatletter
\def\BState{\State\hskip-\ALG@thistlm}
\makeatother

\DeclareMathOperator*{\argmin}{arg\,min}

\newcommand{\M}{\mathcal{M}}

\newcommand{\E}{\mathbb{E}}
\newcommand{\grad}{\operatorname{grad}}

\theoremstyle{plain}
\newtheorem{theorem}{Theorem}[section]
\newtheorem{lemma}[theorem]{Lemma}
\newtheorem{corollary}[theorem]{Corollary}
\newtheorem{proposition}[theorem]{Proposition}

\theoremstyle{definition}
\newtheorem{definition}[theorem]{Definition}
\newtheorem{assumption}{Assumption}

\theoremstyle{remark}
\newtheorem{remark}[theorem]{Remark}

\crefname{theorem}{Theorem}{Theorems}
\crefname{lemma}{Lemma}{Lemmas}
\crefname{corollary}{Corollary}{Corollaries}
\crefname{proposition}{Proposition}{Propositions}
\crefname{definition}{Definition}{Definitions}
\crefname{assumption}{Assumption}{Assumptions}
\crefname{hypothesis}{Hypothesis}{Hypotheses}
\crefname{remark}{Remark}{Remarks}
\crefname{claim}{Claim}{Claims}
\crefname{fact}{Fact}{Facts}
\crefname{algorithm}{Algorithm}{Algorithms}

\title{Proximal DCA for Fr\'echet Regression on Riemannian Manifolds with Bounded Curvature%
\thanks{\textbf{Funding:} This work was funded by the National Science Foundation, Grant \#2443064.}}

\author{%
Yamin Zhou \quad C\'esar A. Uribe\\[0.4em]
\small Department of Electrical and Computer Engineering and Ken Kennedy Institute\\
\small Rice University, Houston, TX\\[0.2em]
\small \texttt{yz282@rice.edu} \quad \texttt{cauribe@rice.edu}%
}

\date{}

\begin{document}

\maketitle

\begin{abstract}
Fr\'echet regression generalizes linear regression to metric-space-valued
responses by defining fitted values as minimizers of weighted Fr\'echet
functionals. Since these weights may have mixed signs, the resulting objective
is a signed barycenter problem rather than a standard convex barycenter
problem. On Riemannian manifolds, this is further complicated by the lack of
global geodesic convexity and possible nonsmoothness of squared distances near
cut loci. We study signed Fr\'echet regression on complete manifolds with
two-sided bounded sectional curvature. By restricting optimization to a
strongly convex normal ball containing the response support, we use local
smoothness, Hessian comparison, and Jacobi-field estimates to formulate the
problem as a locally controlled Riemannian proximal DC problem. This leads to
FRIDA (Fr\'echet Regression via Riemannian Iterative DC Algorithm), an exact and inexact proximal DC algorithm for computing regression
fits. We prove existence and interiority of minimizers under explicit
signed-weight conditions, establish curvature-dependent strong convexity of
the proximal subproblems, and show descent and convergence of the iterates to
stationary points. We also derive sublinear complexity estimates and, under
real-analyticity, obtain full-sequence convergence with KL-type local rates.
These results provide a rigorous optimization framework for signed Fr\'echet
regression on manifolds with bounded curvature.
\end{abstract}

\section{Introduction}

Fr\'echet regression extends classical regression to metric-space-valued
responses by defining fitted values through weighted Fr\'echet minimization.
The global model of Petersen and M\"uller \cite{petersen2019frechet} gives a
conditional Fr\'echet mean at each query point \(x\), and has motivated
extensions to network-valued responses, total-variation regularization,
manifold-valued curve regression, and non-Euclidean predictors
\cite{zhou2022network,lin2021total,torres2022multivariate,nava2024ridge,im2025local}.
Most existing work emphasizes statistical theory, such as consistency,
convergence rates, and model extensions, while the computation of Fr\'echet
regression fits remain comparatively underdeveloped, especially on curved
response spaces where the objective is not geodesically convex.

A key feature of global Fr\'echet regression is that its weights may have mixed
signs. Given observations \(\{(x_i,y_i)\}_{i=1}^m\), the fitted value at a query
point \(x \in \mathbb{R}^q\) is obtained from
\begin{align}\label{eq:main}
   \min_{y \in \mathcal{M}_{\mathrm{ex}}}
   f(y,x) \triangleq \sum_{i=1}^m w_i(x)d^2(y,y_i),
   \qquad
   \sum_{i=1}^m w_i(x)=1,\quad x\in\mathbb{R}^q .
\end{align}
Here \(\mathcal{M}_{\mathrm{ex}}\) is the subset of the manifold on which we
establish existence, and some weights \(w_i(x)\) may be negative. Thus, the
objective is an affine, rather than a convex combination of squared distances.
This signed structure enables extrapolation in predictor space and distinguishes
Fr\'echet regression from ordinary barycenter estimation, but it also changes the
optimization problem: minimizers may fail to exist or be unique, and standard barycenter convexity arguments no longer apply.

\begin{figure}[t!]
    \centering

    \begin{subfigure}[b]{0.24\textwidth}
        \centering
        \includegraphics[
            width=\linewidth,
            trim={7.3cm 7.3cm 6.5cm 4.2cm},
            clip
        ]{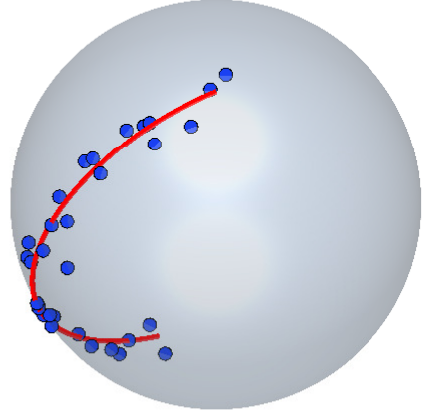}
        \caption{Sphere}
    \end{subfigure}
    \hfill
    \begin{subfigure}[b]{0.24\textwidth}
        \centering
        \includegraphics[
            width=\linewidth,
            trim={6.3cm 7.3cm 5.5cm 4.2cm},
            clip
        ]{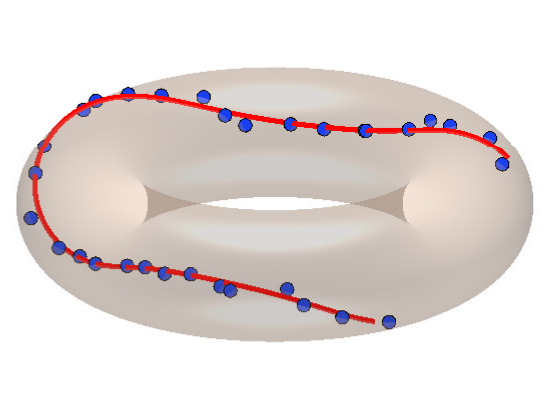}
        \caption{Torus}
    \end{subfigure}
    \hfill
    \begin{subfigure}[b]{0.24\textwidth}
        \centering
        \includegraphics[
            width=\linewidth,
            trim={7.3cm 6.3cm 6.5cm 3.0cm},
            clip
        ]{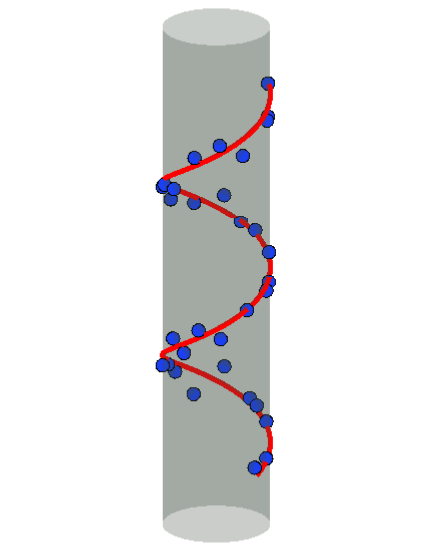}
        \caption{Cylinder}
    \end{subfigure}
    \hfill
    \begin{subfigure}[b]{0.24\textwidth}
        \centering
        \includegraphics[
            width=\linewidth,
            trim={5.3cm 7.3cm 5.5cm 3.2cm},
            clip
        ]{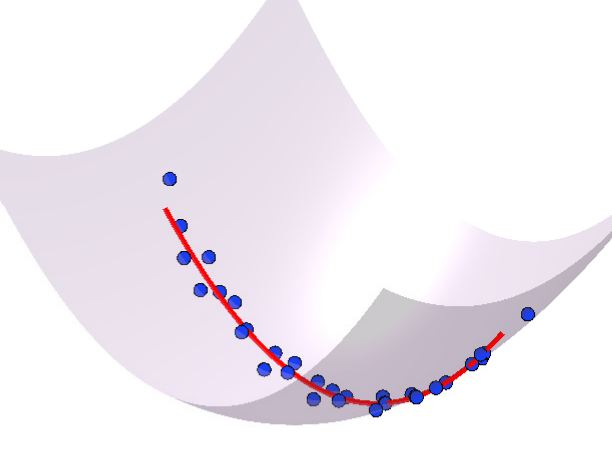}
        \caption{Elliptic paraboloid}
    \end{subfigure}
    \caption{Illustration of Fr\'echet regression on four manifolds: the sphere, torus, cylinder, and elliptic paraboloid. In each panel, the surface (i.e., the manifold) is shown along with a geodesic regressor curve and noisy observations constrained to it. \vspace{-0.6cm}}
    \label{fig:fourmanifolds}
\end{figure}

In Euclidean spaces, signed weighted barycenters have a natural difference-of-convex (DC) structure, linking them to DC programming and DCA literature~\cite{tuy1984global,hiriart1985generalized,horst1999dc,le2018dc,faust2023bregman,yao2023globally,weber2022class}. On Hadamard manifolds, global geodesic convexity of squared distances enabled the authors in~\cite{bergmann2024difference} to develop a Riemannian DCA with iterate convergence to critical points. These results, however, do not cover the positively or nonnegatively curved settings common in Fr\'echet regression. Figure~\ref{fig:fourmanifolds} illustrates representative manifolds where responses follow a geodesic trajectory with noisy observations constrained to the manifold.

This obstruction is not pathological: it already occurs on spheres, where
squared distances are smooth only away from the cut locus and are not globally
geodesically convex; related center-of-mass and gradient-descent issues on
manifolds of constant nonnegative curvature were studied in
\cite{afsari2013convergence}. Similar difficulties arise on other
nonnegatively curved manifolds, such as \(SO(3)\) and real projective spaces
\cite{ziller2014riemannian}. They also appear in optimal transport: the
Bures--Wasserstein manifold of nondegenerate Gaussian covariances has
nonnegative sectional curvature \cite{takatsu2011wasserstein}, yet
squared-distance averaging objectives on it are geodesically nonconvex
\cite{altschuler2021averaging}. Thus, nonnegative curvature naturally yields
regression and barycenter objectives that are smooth on a suitable normal
neighborhoods but lose geodesic convexity beyond a curvature-dependent scale.

This paper develops an optimization framework for signed Fr\'echet regression on complete Riemannian manifolds with controlled sectional curvature. Instead of optimizing over the full manifold, we work on a strongly convex normal ball \(\mathcal M\) containing the response support. On this safe set, logarithm maps are single-valued, squared distances are smooth, Hessian comparison gives local second-order bounds, and Jacobi-field estimates control the linearization of the concave part. We then decompose the signed objective as \(f(y,x)=g(y)-h(y)\) (dependency on $x$ is implicit), where \(g\) and \(h\) collect the positive- and negative-weight terms, and develop FRIDA, a Riemannian proximal DC method adapted to this local geometry.

Our contributions are fourfold. First, we prove the existence and interiority of stationary points for the signed Fr\'echet regression objective on the safe set. We give explicit conditions ensuring that the total negative weight is small enough for the objective to attain a minimum on \(\mathcal M\), with every minimizer lying in \(\operatorname{int}(\mathcal M)\) and hence stationary. We also translate these conditions into guarantees in predictor space, including finite-sample ellipsoidal sufficient conditions.

Second, we derive curvature-dependent estimates for the proposed proximal DC method under two-sided sectional-curvature bounds. The curvature upper bound controls the lower bounds on the Hessian for squared distances and hence the convexity radius, while the curvature lower bound controls the Jacobi-field growth and the logarithm-Hessian smoothness constants. Together, these estimates yield computable lower bounds that ensure strong geodesic convexity of the proximal DC subproblems on a subset of the manifold.

Third, we propose exact and inexact versions of FRIDA using these estimates. The proximal DC subproblems are solved directly on the manifold, preserving feasibility by construction. We prove objective descent, boundedness of the iterates in the safe set, and stationarity of every accumulation point for the original signed Fr\'echet regression objective. We also derive explicit \(O((N+1)^{-1/2})\) sublinear complexity bounds for the smallest successive-step distance among the first \(N+1\) iterations.

Fourth, under an additional real-analyticity assumption on the manifold, we strengthen the convergence analysis via the Kurdyka--\L ojasiewicz framework. In that case, the whole sequence converges to a single stationary point, and the local rate is determined by the KL exponent: finite termination when the exponent is zero, linear convergence for exponents in \((0,1/2]\), and sublinear convergence otherwise.

The paper contributes to the intersection of Fr\'echet regression and Riemannian optimization. It provides an algorithmic foundation for evaluating global Fr\'echet regression with signed weights on curved response spaces, and develops FRIDA as a proximal DC framework for signed barycenter problems outside the Hadamard setting, where explicitly local geometric control is
required.

The rest of the paper is organized as follows. Section~\ref{sec:background} presents preliminary background and notation. Section~\ref{sec:problem formulation} formulates the signed Fr\'echet regression problem and introduces the safe-set geometry. Section~\ref{sec:existence} establishes existence and interiority of minimizers. Section ~\ref{sec:converg analysis} shows the convergence analysis for the main convergence guarantees of exact and inexact FRIDA. Improved KL-based rates are shown in Section~\ref{subsec:analytic_rates}. Numerical analysis is shown in Section~\ref{sec:numerics}.

\section{Background and notation}\label{sec:background}

We work on a complete Riemannian manifold~\((\Omega,d)\) with geodesic distance \(d\). For \(x\in\Omega\), \(T_x\Omega\) denotes the tangent space, equipped with inner product \(\langle\cdot,\cdot\rangle_x\) and norm \(\|\xi\|_x=\sqrt{\langle \xi,\xi\rangle_x}\). In common normal neighborhoods, we use the exponential map \(\exp_x\), its inverse
\(\log_x\), and parallel transport \(P_{x\to y}\).
\begin{definition}[Eq.(1.3)--(1.6) in \cite{viaclovsky_pcmi_curvature}]\label{def:curvature_tensor}
Let \(R\) denote the Riemann curvature tensor on \(\Omega\), with convention $R(X,Y)Z=\nabla_X\nabla_Y Z-\nabla_Y\nabla_X Z-\nabla_{[X,Y]}Z$, and let \(\nabla R\) denote its covariant derivative.
\end{definition}

For a \(C^1\) function \(f:\Omega\to\mathbb R\), the Riemannian gradient \(\grad f(x)\in T_x\Omega\) is defined by 
$$Df(x)[\xi]=\langle \grad f(x),\xi\rangle_x,$$ for all $\xi\in T_x\Omega$. Whenever \(z\) and \(y\) lie in a common normal neighborhood, $\grad_y (1/2) d^2(z,y)=-\log_y(z)$.

\begin{definition}Let \(\mathcal X\subset \Omega\) be geodesically convex and let \(f:\mathcal X\to\mathbb R\) be \(C^1\).
We say that \(f\) is \(\mu\)-strongly geodesically convex on \(\mathcal X\) if
\[
f(y)\ge f(x)+\langle \grad f(x),\log_x(y)\rangle_x+\frac{\mu}{2}d^2(x,y),
\qquad \forall x,y\in\mathcal X.
\]
We say that \(f\) is \(L\)-smooth on \(\mathcal X\) if $
\|\grad f(y){-}P_{x\to y}\grad f(x)\|_y{\le} L\,d(x,y)$, $\forall$ $x,y\in\mathcal X$.
\end{definition}

% A standard consequence of \(L\)-smoothness is the descent estimate
% \[
% f(\exp_x(-t\grad f(x)))
% \le
% f(x)-\big(t-({L}/{2})t^2\big)\|\grad f(x)\|_x^2
% \]
% for every \(t\ge0\) such that \(\exp_x(-t\grad f(x))\in\mathcal X\)~\cite[Corollary~2.1]{bento2016iteration}.

\begin{definition}\label{def:comparison_constants_short}
Let \(\mathcal X\subset\Omega\) be geodesically convex with
\(\operatorname{diam}(\mathcal X)\le D\), and assume
\(-\Lambda_- \leq \sec_\Omega \leq \Lambda_+\) on \(\mathcal X\), where
\(\Lambda_-,\Lambda_+\ge 0\). Define
$\delta_D\triangleq\delta_+(D), \zeta_D\triangleq\zeta_-(D)$
\[
\delta_+(t)\triangleq
\begin{cases}
t\sqrt{\Lambda_+}\cot(t\sqrt{\Lambda_+}), & \Lambda_+>0,\\
1, & \Lambda_+=0,
\end{cases}
\qquad
\zeta_-(t)\triangleq
\begin{cases}
t\sqrt{\Lambda_-}\coth(t\sqrt{\Lambda_-}), & \Lambda_->0,\\
1, & \Lambda_-=0.
\end{cases}
\]
Also, define the logarithm-Hessian factors as
\[
\textstyle
\alpha_+(t)\triangleq\frac{t\sqrt{\Lambda_+}}{\sin(t\sqrt{\Lambda_+})},\quad
b_-(t)\triangleq\frac{\sinh(t\sqrt{\Lambda_-})}{t\sqrt{\Lambda_-}},\quad
c_-(t)\triangleq\cosh(t\sqrt{\Lambda_-}),
\]
with limiting value \(1\) when \(\Lambda_\pm=0\); here
\(\delta_+,\zeta_-\) are the squared-distance Hessian-comparison constants.
\end{definition}

\begin{lemma}[Lemma 23 in~\cite{martinez2024convergence}]\label{lem:sqdist_curv}
Let \(\mathcal X\subset \Omega\) be geodesically convex with
\(\operatorname{diam}(\mathcal X)\le D\), and assume
$-\Lambda_- \le \sec_\Omega \le \Lambda_+$ on $\mathcal X$. 
If \(\Lambda_+>0\), assume also \(D<\pi/(2\sqrt{\Lambda_+})\).
Then, for all \(z,y\in\mathcal X\) and all \(v\in T_y\Omega\),
\[
\delta_D \|v\|_y^2
\le
\operatorname{Hess}_y\!\left(\frac12 d^2(z,y)\right)[v,v]
\le
\zeta_D \|v\|_y^2 .
\]
Consequently, \(y\mapsto \frac12 d^2(z,y)\) is
\(\delta_D\)-strongly convex and \(\zeta_D\)-smooth on \(\mathcal X\).
\end{lemma}

\begin{definition}
Let \(\gamma:[0,\ell]\to \Omega\) be a geodesic. A vector field \(J\) along \(\gamma\) is a Jacobi field if $$D_t^2J+R(J,\dot\gamma)\dot\gamma=0,$$ where \(D_t\) denotes the covariant derivative along~\(\gamma\). If \(\gamma(t)=\exp_p(tu)\) and
\(w\in T_p\Omega\), then the Jacobi field satisfying \(J(0)=0\) and \(D_tJ(0)=w\) is given by $J(t)=d(\exp_p)_{tu}(tw)$~\cite[Chapter~5, Proposition~2.7]{do1992riemannian}.
\end{definition}

\begin{definition} [Ch.~1 in~\cite{Lee_2013}]
A Riemannian manifold is called real analytic if it admits an atlas with real-analytic transition maps.
\end{definition}

\begin{definition}\label{def:kl}
Let \(\Omega\) be a Riemannian manifold and let \(f:\Omega \to\mathbb R\) be \(C^1\).
We say that \(f\) satisfies the Riemannian Kurdyka--\L{}ojasiewicz property at \(\bar x\in\Omega\) if there exist
a neighborhood \(U\) of \(\bar x\), \(\delta>0\), and a concave function \(\varphi:[0,\delta)\to[0,\infty)\) such that
\(\varphi(0)=0\), \(\varphi\in C^1(0,\delta)\), \(\varphi'(s)>0\) on \((0,\delta)\), and $\varphi'(f(x)-f(\bar x))\,\|\grad f(x)\|_x \ge 1$ for all \(x\in U\) satisfying \(f(\bar x)<f(x)<f(\bar x)+\delta\).
\end{definition}

\begin{proposition}[Theorem~3.5 in~\cite{hosseini2015convergence} and Section~9 in~\cite{kurdyka2000proof}]\label{prop:analytic_kl}
Let $\Omega$ be a real-analytic Riemannian manifold and let $f:\Omega\to\mathbb R$ be real analytic. Then $f$ is locally Lipschitz and subanalytic; in particular, it is a locally Lipschitz $\mathcal C$-function. Hence $f$ satisfies the Riemannian Kurdyka--{\L}ojasiewicz property at every point of $\Omega$.
\end{proposition}
\section{Problem Formulation} \label{sec:problem formulation}

Let $(\Omega,d)$ be an \(n\)-dimensional complete Riemannian manifold with bounded sectional curvature, and consider a random pair $(X,Y)\sim \mathcal{P}$, where $X\in \mathbb{R}^q$ and $Y \in \Omega$ and $\mathcal{P}$ is a joint distribution on $\mathbb{R}^q \times \Omega$. Additionally, assume the marginal distributions $X\sim \mathcal{P}_X$ and $Y\sim \mathcal{P}_Y$ exist. Moreover, assume $\mu = \mathbb{E}[X]$ and $\Sigma = \text{Var}(X)$ exist with $\Sigma \succ 0$. The Fr\'echet regression function~\cite{petersen2019frechet} of $Y$ given $X=x$, also known as the conditional barycenter of $Y$ given $X=x$, is defined as
\begin{align}\label{eq:main_frechet}
    m(x) &=  \argmin_{y \in \Omega} M(x,y) \triangleq  \mathbb{E}[d^2(Y,y)\mid X=x].
\end{align}
Classically~\eqref {eq:main_frechet} is solved by approximating it into a system of unconditional expectations~\cite{hansen1982large,bennett2023variational,lin2024type}. Following the global Fr\'echet regression construction of \cite{petersen2019frechet}, one can define the (unconstrained) global Fr\'echet regression function by
\begin{align}\label{eq:glob_frec}
m_{\oplus}(x)
  \triangleq  \argmin_{y\in\Omega} M_{\oplus}(y,x),
  \qquad
M_{\oplus}(y,x)
  \triangleq  \mathbb E\!\left[s(X,x)d^2(Y,y)\right],
\end{align}
where $s(z,x)=1+(z-\mu)^\top\Sigma^{-1}(x-\mu)$, and $\mathbb E[s(X,x)]=1$. Moreover, for every fixed \(x\in\mathbb R^q\), $\mathbb E|s(X,x)| < \infty$. In Euclidean spaces, Eq.~\eqref{eq:glob_frec} reproduces ordinary linear regression; on a general metric space, it should be viewed as a model-based global approximation to the conditional Fréchet mean.

In the sequel, we study the constrained problem
\begin{align}
m_{\oplus}(x)
  \triangleq  \argmin_{y\in\mathcal M_{\mathrm{ex}}} M_{\oplus}(y,x),
\end{align}
where \(\mathcal M_{\mathrm{ex}}\subset\Omega\) is a safe set where solutions are guaranteed to exist even in the presence of negative weights.

Similarly, if instead of the joint distribution $\mathcal{P}$, one has access to independent realizations of $(x_i,y_i)$ for $i=1,\cdots,m$, with positive definite sample covariance, the estimator of the linear Fr\'echet regression function is defined as
\begin{align}
     \hat{m}_{\oplus}(x) &= \argmin_{y \in \mathcal{M}_{\mathrm{ex}}}  \frac{1}{m} \sum_{i=1}^m s_{i,m}(x) d^2(y_i,y) \ \text{with} \ \  \frac{1}{m}\sum_{i=1}^m s_{i,m}(x)  =1, \ \forall x \in \mathbb{R}^q, \nonumber
     \\   s_{i,m}(x)  & = 1 + (x_i {-}\hat{\mu})^\top{\hat{\Sigma}^{-1}}(x{-}{\hat\mu}), \   \hat{\mu} = \frac{1}{m}\sum_{i=1}^m x_i, \nonumber  
      \hat{\Sigma}   = \frac{1}{m} \sum_{i=1}^m(x_i-\hat{\mu})(x_i-\hat{\mu})^T  \nonumber
\end{align}

The definition of the estimator $\hat{m}_{\oplus}(x)$ provides the function class of problems we will focus on in this paper. Specifically, for a given set of pairs $\{(x_i,y_i)\}_{i=1}^m$, where $x_i \in \mathbb{R}^q$ and $y_i \in \Omega$, we will focus on solving optimization problems of the form~Eq.~\eqref{eq:main} where and \(w_i(x)=m^{-1}s_{i,m}(x)\) are possibly negative weights determined by the query point \(x\) and the observed covariates \(\{x_i\}_{i=1}^m\). For the sake of simplicity in the subsequent analysis, we introduce the following notation for Problem~\eqref{eq:main}, separating the summands into those with positive and negative weights.
\begin{align}\label{eq:main2}
   & \min_{y \in \mathcal{M}_{\mathrm{ex}}} f(y,x) \triangleq  g(y) - h(y), \ \text{where} \\
   & g(y) = \sum_{w_i(x)\geq 0}w_{i}(x)d^{2}(y,y_i), \ \text{and} \  \ 
    h(y) =- \sum_{w_{j}(x)<0}w_j (x)d^{2}(y,y_j). \nonumber
\end{align}
Note that implicit dependencies on $x$ are omitted but should be clear from context. Moreover, for a given $x\in\mathbb{R}^q$, we define $$w_{+}(x) = \sum_{w_i(x)\geq 0}w_i(x), \ \  \text{and} \ \  w_{-}(x) = -\sum_{w_j(x)< 0}w_j(x).$$

\begin{remark}[Global and local FRIDA]\label{rem:global_local_frida}
The weights \(w_i(x)=m^{-1}s_{i,m}(x)\) defined above are the global affine
weights of Fr\'echet regression.  When Algorithm~\ref{alg:rpdca}
is applied with these weights, we call the resulting method \emph{global FRIDA}.

The same framework also covers local Fr\'echet regression by replacing the
global weights with the local-linear weights in~\cite{petersen2019frechet}. For $j=0,1,2,$
\[
    w^{\mathrm{loc}}_{i,m}(x)
    =
    \frac{1}{m}K_h(x_i-x)
    \frac{\mu_2(x)-\mu_1(x)(x_i-x)}
    {\mu_0(x)\mu_2(x)-\mu_1(x)^2},
    \mu_j(x)
    =
    \frac{1}{m}\sum_{i=1}^m K_h(x_i-x)(x_i-x)^j,
\]
\(K_h(u)=h^{-1}K(u/h)\) for a bandwidth \(h>0\).With standard kernels such as Gaussian, Epanechnikov, or quartic
\cite{MR1383587}, Algorithm~\ref{alg:rpdca} with these weights is called
\emph{local FRIDA}.

Thus, global and local FRIDA differ only in the weights: after recomputing
\(w_+(x)\) and \(w_-(x)\), the objective, DC splitting, and algorithm are
unchanged. The admissible region may change, but FRIDA remains agnostic to the
weight construction.
\end{remark}
\begin{remark}
    Problem \eqref{eq:main} is not a barycenter computation based on a convex combination of squared distances; it is an affine combination, since some weights may be negative. In fact, it is common to have negative weights. For example, if the $x_i$ sampled from $X$ are $x_1 = 0, x_2 = 2, x_3 = 4$ and set $x = {1}/{2}$, then the corresponding weights are: $w_1 = {17}/{24}, w_2 = {1}/{3}, w_3 = -{1}/{24}$.  Generally, $w_i<0$ if and only if $(x_i-\hat\mu)^\top \hat{\Sigma}^{-1} (x-\hat{\mu})<-1$. A simple sufficient condition for nonnegativity of all weights follows from
Cauchy--Schwarz. Define $D \triangleq  \max_{1\le i\le m}\|\hat\Sigma^{-1/2}(x_i-\hat\mu)\|_2$. Then $\|\hat\Sigma^{-1/2}(x-\hat\mu)\|_2 \le {1}/{D}
$ and $w_i(x)\ge 0,\ \forall i$. Depending on the geometry of the observed covariates, \(D\) can be large, in which case the sufficient nonnegativity region becomes small and negative weights may occur even for moderate values of \(\|x\|_2\). Formally, the following lemma specifies the range of $x$ values that yield negative weights for arbitrary data points in real space.
\end{remark}
\begin{lemma}
Let $x_1,\dots,x_m\in\mathbb{R}^{q}$ and without loss of generality, let $\hat{\mu} = 0$, $\hat{\Sigma} = I$. Then, $$w_i(x)\geq 0, \ \forall x \in \mathbb{X} \subseteq \mathbb{R}^q,$$ where $\mathbb{X} =\bigcap_{i=1}^m \{\,x\in\mathbb{R}^q:\ x_i^\top x\ge -1\,\}$. 
\end{lemma}
\begin{proof}
In the whitened coordinates \(\hat\mu=0\) and \(\hat\Sigma=I\), one has $w_i(x)=(1/m)(1+x_i^\top x)$. Hence \(w_i(x)\ge 0\) if and only if \(x_i^\top x\ge -1\). Imposing this for all
\(i=1,\dots,m\) yields the stated intersection of half-spaces.
\end{proof}

The convergence analysis has three steps. We first identify a compact, well-posed region for the affine weighted objective. Second, we construct a proximal DC model that is strongly geodesically convex on smaller local balls. Third, we combine the resulting descent estimate with compactness to prove stationarity of accumulation points.

\section{Existence of Solutions, FRIDA and Main Results}\label{sec:existence}

\subsection{Existence of a Minimizer on a Safe Set}

In this section, we isolate a compact geodesic ball containing all responses and show that the affine weighted objective attains a minimum there. A boundary-gradient condition then rules out boundary minimizers, so every minimizer is interior and therefore stationary. Strong convexity will only be needed later, on smaller local balls used in the proximal subproblems. We then specialize this criterion to finite-sample and ellipsoidal covariate regions.

\begin{assumption}\label{assum:bounded_curv}
     The Riemannian manifold \((\Omega,d)\) is complete and its sectional curvatures satisfy: $$-\Lambda_- \le \sec_\Omega \le \Lambda_+ < \infty,  \ \ \Lambda_-,\Lambda_+\ge 0.$$
\end{assumption}

\begin{assumption}\label{assum:bound_y}
      There exists $c\in\Omega$ and $r>0$ such that $\mathbb{P}\big(Y\in B_r(c)\big)=1$.
\end{assumption}

\begin{assumption}\label{assum:bound_ex}
    There exists $\rho_{\mathrm{ex}}>0$ such that $r{<}\rho_{\mathrm{ex}}{<} \min\{\iota_{r,c},\pi/\sqrt{\Lambda_+}\}- r$.
\end{assumption}

\begin{assumption}\label{assum:rho}
    There exists $\rho{>}0$ such that $r{<} \rho {<}  \min\{(1/2)\iota_{\rho_{\mathrm{ex}},c}, \pi/ (2\sqrt{\Lambda_+})\}$.
\end{assumption}
Assumptions~\ref{assum:bound_y}--\ref{assum:bound_ex} define the
existence-safe ball
\(\M_{\rm ex}\triangleq\overline{B_{\rho_{\rm ex}}(c)}\), which contains
\(\M_r\triangleq\overline{B_r(c)}\) with a positive injectivity margin. Indeed, for
any \(z\in\M_r\) and \(y\in\M_{\rm ex}\),
\(d(z,y)\le r+\rho_{\rm ex}<\iota_{r,c}\le \operatorname{inj}_\Omega(z)\). Thus, \( d^2(z,y)\), and \(M_\oplus(\cdot,x)\) are smooth on
\(\M_{\rm ex}\) with respect to $y$. Assumption~\ref{assum:rho} introduces the smaller
algorithmic ball \(\M_\rho\triangleq\overline{B_\rho(c)}\). Since
\(2\rho<\iota_{\rho,c}\) and \(\rho<\pi/(2\sqrt{\Lambda_+})\),
\(\M_\rho\) is a strongly convex normal ball~\cite[Lemma~3.2.1]{wintraecken2015ambient};
in particular, \(\log_x(y)\) is well defined for all
\(x,y\in\M_\rho\), and every two points of \(\M_\rho\) are joined by a
unique minimizing geodesic in \(\M_\rho\). Finally,
Assumption~\ref{assum:bounded_curv} gives the two-sided curvature control
used below: \(\Lambda_+\) controls the lower Hessian-comparison constants,
while \(\Lambda_-\) controls the upper smoothness constants. Set
\[
\delta_{\rm ex}\triangleq\delta_+(r+\rho_{\rm ex}),\ \
\zeta_{\rm ex}\triangleq\zeta_-(r+\rho_{\rm ex}),\ \
L_R\triangleq\sup_{q\in\M_{\rm ex}}\|(\nabla R)_q\|,\ \
\Lambda_0\triangleq\max\{\Lambda_+,\Lambda_-\}.
\]
Since \(\M_{\rm ex}\) is compact and \(R,\nabla R\) are smooth, \(L_R<\infty\), and \(\|R_q\|\le c_n\Lambda_0\) on \(\M_{\rm ex}\)
for some \(c_n>0\) depending only on \(n\) and the tensor norm~\cite[Proposition~1.1]{viaclovsky_riemannian_geometry_2011}.

\begin{theorem}\label{thm:main_safe}
Let Assumptions~\ref{assum:bounded_curv}--\ref{assum:bound_ex} hold, for every fixed \(x\in\mathbb R^q\) such that \(\E|s(X,x)|<\infty\), the objective \(M_\oplus(\cdot,x)\) attains a minimum on \(\mathcal M_{\mathrm{ex}}\). If, in addition,
\begin{align}\label{cond:weights_general}
    \E[(s(X,x))_-] < {(\rho_{\mathrm{ex}}-r)}/{(2r)},
\end{align}
then every minimizer \(y^\star\in \operatorname{int}(\mathcal M_{\mathrm{ex}})\) and is stationary, i.e., $\grad_y M_\oplus(y^\star,x)=0$.
\end{theorem}

\begin{proof}
We first show that \(M_\oplus(\cdot,x)\) is well defined and attains a minimum on~\(\mathcal M_{\mathrm{ex}}\).
By Assumption~\ref{assum:bound_y}, \(Y\in B_r(c)\) almost surely, so for every \(y\in \mathcal M_{\mathrm{ex}}\),
\[d(Y,y)\le d(Y,c)+d(c,y)\le r+\rho_{\mathrm{ex}}
\qquad \text{a.s.}
\]
Hence $|s(X,x)|\,d^2(Y,y)\le (r+\rho_{\mathrm{ex}})^2 |s(X,x)|$, and the right-hand side is integrable. Therefore \(M_\oplus(y,x)\) is finite for every \(y\in \mathcal M_{\mathrm{ex}}\). For every \(z\in \M_r\), the map \(d^2(z,y)\) is smooth on \(\mathcal M_{\mathrm{ex}}\), and
\[
\sup_{z\in \M_r\,y\in\mathcal M_{\mathrm{ex}}}
\|\grad_y d^2(z,y)\|_y
\le 2(r+\rho_{\mathrm{ex}})<\infty.
\]
Thus, dominated convergence allows differentiation under the expectation, so $M_\oplus(\cdot,x)\in C^1(\mathcal M_{\mathrm{ex}})$.
Since \((\Omega,d)\) is complete, Hopf--Rinow implies that the closed bounded ball \(\mathcal M_{\mathrm{ex}}\) is compact. Hence \(M_\oplus(\cdot,x)\) attains a minimum on \(\mathcal M_{\mathrm{ex}}\).

Now let \(y\in \partial\mathcal M_{\mathrm{ex}}\), and let \(\nu_y\triangleq \grad d(c,\cdot)|_y\) be the outward unit normal.
By \cite[Lem.~3.4.8]{wintraecken2015ambient}, for every \(z\in \M_r\), $
\langle -\log_y(z),\nu_y\rangle_y \ge \rho_{\mathrm{ex}}-r$. Also,
\[
\|-\log_y(z)\|_y=d(y,z)\le \rho_{\mathrm{ex}}+r.
\]
\[
\text{Therefore}\ \ \ \langle \grad_y M_\oplus(y,x),\nu_y\rangle_y
\ge
2(\rho_{\mathrm{ex}}-r)\E[(s(X,x))_+]
-
2(\rho_{\mathrm{ex}}+r)\E[(s(X,x))_-].
\]
Since \(\E[(s(X,x))_+]-\E[(s(X,x))_-]=1\), condition~\eqref{cond:weights_general} implies
\[
(\rho_{\mathrm{ex}}-r)\E[(s(X,x))_+]-(\rho_{\mathrm{ex}}+r)\E[(s(X,x))_-]
=
\rho_{\mathrm{ex}}-r-2r\,\E[(s(X,x))_-]
>0.
\]
Thus, the outward directional derivative is strictly positive on \(\partial\mathcal M_{\mathrm{ex}}\), so no minimizer can lie on the boundary. Hence, every minimizer \(y^\star\) lies in \(\operatorname{int}(\mathcal M_{\mathrm{ex}})\), and the first-order necessary condition gives $$\grad_y M_\oplus(y^\star,x)=0.$$
\end{proof}

Theorem~\ref{thm:main_safe} shows that if the total negative weight is not too large, then minimizers exist and cannot occur on the boundary of the safe ball.

\begin{corollary}\label{cor:moments}
Let Assumptions~\ref{assum:bounded_curv},~\ref{assum:bound_y}, and~\ref{assum:bound_ex} hold. If
\[
x\in \mathcal X_{\mathrm{ex}},
\qquad
\mathcal X_{\mathrm{ex}}
\triangleq 
\left\{
x\in \mathbb{R}^q:
(x-\mu)^\top\Sigma^{-1}(x-\mu)
<
\left({\rho_{\mathrm{ex}}}/{r}\right)^2-1
\right\},
\]
then \(M_{\oplus}(\cdot,x)\) attains a minimum in \(\operatorname{int}(\mathcal M_{\mathrm{ex}})\) and is stationary.
\end{corollary}

\begin{corollary}\label{cor:finite_exact}
Let \(\{(x_i,y_i)\}_{i=1}^m\) be \(m\) i.i.d. realizations of \((X,Y)\sim\mathcal P\), and let Assumptions~\ref{assum:bounded_curv},~\ref{assum:bound_y}, and~\ref{assum:bound_ex} hold. For every fixed \(x\in\mathbb R^q\), if $w_-(x)<{(\rho_{\mathrm{ex}}-r)}/{(2r)}$, then \(f(\cdot,x)\) in \eqref{eq:main2} has an interior minimizer in \(\mathcal M_{\mathrm{ex}}\); every minimizer is stationary. Moreover, the same result follows if alternatively
\[
x\in \widehat{\mathcal X}_{\mathrm{ex}},
\qquad
\widehat{\mathcal X}_{\mathrm{ex}}
\triangleq 
\left\{
x\in\mathbb R^q:
(x-\hat\mu)^\top\hat\Sigma^{-1}(x-\hat\mu)
<
\left({\rho_{\mathrm{ex}}}/{r}\right)^2-1
\right\}
\]
\end{corollary}

% \begin{corollary}\label{cor:finite_var}
% Let \(\{(x_i,y_i)\}_{i=1}^m\) be \(m\) i.i.d. realizations of \((X,Y)\sim\mathcal P\), and let
% Assumptions~\ref{assum:bounded_curv},~\ref{assum:bound_y}, and~\ref{assum:bound_ex} hold.
% If
% \[
% x\in \widehat{\mathcal X}_{\mathrm{ex}},
% \qquad
% \widehat{\mathcal X}_{\mathrm{ex}}
% \triangleq 
% \left\{
% x\in\mathbb R^q:
% (x-\hat\mu)^\top\hat\Sigma^{-1}(x-\hat\mu)
% <
% \left({\rho_{\mathrm{ex}}}/{r}\right)^2-1
% \right\},
% \]
% then \(f(\cdot,x)\) in \eqref{eq:main2} has an interior minimizer in \(\mathcal M_{\mathrm{ex}}\); every minimizer is stationary.
% \end{corollary}

% \begin{proof}
% Let \(s_i\triangleq s_{i,m}(x)\) and recall \(w_i(x)=m^{-1}s_i\). Since \(m^{-1}\sum_{i=1}^m s_i=1\),
% \[
% w_-(x)=\frac1m\sum_{i=1}^m (s_i)_-
% =\frac12\left(\frac1m\sum_{i=1}^m |s_i|-1\right)
% \le \frac12\left(\sqrt{\frac1m\sum_{i=1}^m s_i^2}-1\right).
% \]
% Write
% \[
% u=\hat\Sigma^{-1/2}(x-\hat\mu),
% \qquad
% z_i=\hat\Sigma^{-1/2}(x_i-\hat\mu),
% \]
% so that \(s_i=1+z_i^\top u\). Since
% \[
% \frac1m\sum_{i=1}^m z_i=0,
% \qquad
% \frac1m\sum_{i=1}^m z_i z_i^\top=I,
% \]
% we obtain
% \[
% \frac1m\sum_{i=1}^m s_i^2
% =
% \frac1m\sum_{i=1}^m (1+z_i^\top u)^2
% =
% 1+\|u\|^2
% =
% 1+(x-\hat\mu)^\top\hat\Sigma^{-1}(x-\hat\mu).
% \]
% Hence, if
% \[
% (x-\hat\mu)^\top\hat\Sigma^{-1}(x-\hat\mu)
% <
% \left(\frac{\rho_{\mathrm{ex}}}{r}\right)^2-1,
% \]
% then
% \[
% w_-(x)
% <
% \frac12\left(\frac{\rho_{\mathrm{ex}}}{r}-1\right)
% =
% \frac{\rho_{\mathrm{ex}}-r}{2r}.
% \]
% The conclusion follows from Corollary~\ref{cor:finite_exact}.
% \end{proof}

 Corollary~\ref{cor:moments} gives a convenient sufficient condition directly in covariate space: an explicit ellipsoidal region where existence and interior stationarity are guaranteed on~\(\mathcal M_{\mathrm{ex}}\). Corollary~\ref{cor:finite_exact} is the exact finite-sample analog of Theorem~\ref{thm:main_safe}: the population negative-weight condition is replaced by its empirical counterpart, and the conclusion holds on~\(\mathcal M_{\mathrm{ex}}\).  It also provides a directly computable safe extrapolation region from the observed covariates for the larger existence-safe ball \(\mathcal M_{\mathrm{ex}}\).
 
The results in this section identify concrete geometric regions in covariate space where minimizers are guaranteed to exist and, moreover, are forced to lie strictly inside the existence-safe ball \(\mathcal M_{\mathrm{ex}}\), we thus define $f_*\triangleq \min_{y\in \mathcal M_{\mathrm{ex}}}f(y,x)$.

\subsection{FRIDA: Riemannian Iterative DC Algorithm and Main Results}

The previous subsection guarantees that the objective is well defined on the existence-safe set \(\mathcal M_{\mathrm{ex}}\). We now introduce the method FRIDA:  Fréchet Regression via Riemannian Iterative DC Algorithm, which builds upon a proximal DC step on an adaptive local ball \(\mathcal M_k\subset \mathcal M_{\mathrm{ex}}\), chosen so that the linearization of the concave part remains controlled and the proximal model is strongly geodesically convex.

\begin{algorithm}[htb!]
\caption{FRIDA: Fréchet Regression via Riemannian Iterative DC Algorithm}
\label{alg:rpdca}
\begin{algorithmic}[1]
\REQUIRE  \(y_0\in\operatorname{int}(\M_{\rm ex})\), 
\(\zeta=1/4\), \(x\in\mathbb R^q\), \(\rho>0\), curvature constants
\(\Lambda_\pm,L_R,c_n\), \(\epsilon_k>0\) with
\(\sum_{k=0}^{\infty}\epsilon_k<\infty\), \(\theta\in(0,1)\), \(\eta_0>0\),
and \(\delta_{\rm ex}\triangleq\delta_+(r+\rho_{\rm ex})\).
\FOR{$k=0,1,2,\ldots$}
    \STATE Choose $r_k=\min\{\theta\,\operatorname{dist}(y_k,\partial\mathcal M_\mathrm{ex}),\rho\}$, and $\M_k = \overline{B_{r_k}(y_k)}$.
    \STATE Set \(\delta_k\triangleq\delta_+(r_k)\) and $L_{\log}^{\pm}(r_k)\triangleq \alpha_+(r_k)^3\left[\frac16 L_Rr_k^2 b_-(r_k)^3+\frac56 c_n\Lambda_0r_k b_-(r_k)^2c_-(r_k)\right].$
    \STATE Set $\tau_k= \max\left\{\frac{L_{\log}^{\pm}(r_k)\|\grad h(y_k)\|-2w_+(x)\delta_{\mathrm{ex}}}{\delta_{+}(r_k)}  + \frac{2\|\grad f(y_k)\|}{\delta_{+}(r_k) r_k}     +\eta_0,\ 1-2w_-(x)\delta_{\mathrm{ex}}   , 1 \right\}.$
     \STATE Compute $y_{k+1}$ according to one of the following cases:
    \STATE \textbf{Exact Method:} 
    \STATE \hspace{5mm} $y_{k+1} = \argmin_{y \in \mathcal{M}_k} \Phi_k(y) \triangleq  \Bigl\{ g(y) - \langle\grad h(y_k),\log_{y_k}(y)\rangle + \frac{\tau_k}{2}d^{2}(y_k, y)\Bigr\}$ \label{def:phi}
    \STATE  \textbf{Inexact Method:} 
    \STATE  Find $\hat{y} \in \M_k$ s.t. $\|\grad \Phi_{k}(\hat{y})\| \;\le\; \min\big(\epsilon_k, \zeta d(y_k,\hat y) \big)$, and $\Phi_k(\hat{y})\leq \Phi_{k}(y_k)$
    \STATE   \hspace{5mm} $y_{k+1} = \hat{y}$.
\ENDFOR
\end{algorithmic}
\end{algorithm}

\begin{theorem}\label{thm:accumulation}
Let \(\{(x_i,y_i)\}_{i=1}^m\) be \(m\) i.i.d. realizations of \((X,Y)\sim\mathcal P\), and let
Assumptions~\ref{assum:bounded_curv}--\ref{assum:rho} hold.
Fix \(x\in\mathbb R^q\) such that $w_-(x)<{(\rho_{\mathrm{ex}}-r)}/{(2r)}$. Assume that $y_0\in \operatorname{int}(\mathcal M_{\mathrm{ex}})$ and $f(y_0,x)<\min_{y\in\partial\mathcal M_{\mathrm{ex}}} f(y,x)$. Let \(\{y_k\}\) be the sequence generated by Algorithm~\ref{alg:rpdca}. Then:

\begin{enumerate}
    \item The sequence \(\{y_k\}\subset \mathcal M_{\mathrm{ex}}\)  has at least one accumulation point \(\bar y\in \mathcal M_{\mathrm{ex}}\)
    \item Every accumulation point \(\bar y\) is stationary for \(f(\cdot,x)\), i.e. $\grad f(\bar y,x)=0$.
    \item For every \(N\in\mathbb N\), the following estimates hold:
\end{enumerate}
    \begin{align}
        &\min_{0\le k\le N} d(y_k,y_{k+1})
        \le
        \sqrt{\frac{f(y_0,x)-f_*}{\kappa(N+1)}}, \quad  \kappa =
\begin{cases}
1/2, & \text{Exact Method}, \\
1/4, &  \text{Inexact Method}
\end{cases}
        \label{main:exact}        
    \end{align}
\end{theorem}

The strict sublevel assumption on \(y_0\) ensures that all iterates remain in a compact subset of \(\operatorname{int}(\mathcal M)\), which in turn provides a uniform positive distance from the boundary and allows the proximal parameters \(\tau_k\) to be chosen uniformly bounded.

\section{Convergence Analysis }\label{sec:converg analysis}

The proof uses three ingredients: (i) Hessian control of the linearized concave part, (ii) strong convexity of proximal subproblems, and (iii) a decrease estimate yielding vanishing steps and stationarity of accumulation points.

\subsection{Subproblem Strong Convexity and Closed Iterations}

Next, we provide the local geometry needed by the proximal model. The first controls the Hessian of the linearized concave term
\(y\mapsto \langle \xi,\log_p(y)\rangle\), while the second gives lower and upper Hessian bounds for the positive and negative weighted squared-distance terms.

\begin{lemma}\label{lemma:bound_hess_lin}
Let Assumptions~\ref{assum:bounded_curv}, \ref{assum:bound_ex}, and \ref{assum:rho} hold. 
Fix \(p\in \operatorname{int}(\mathcal M_{\mathrm{ex}})\), \(\xi\in T_p\Omega\), and let $$0<r_p<\min\{\operatorname{dist}(p,\partial\mathcal M_{\mathrm{ex}}),\,\rho\}.$$ 
For \(y\in \overline{B_{r_p}(p)}\setminus\{p\}\), define $$\psi(y)\triangleq \langle \xi,\log_p(y)\rangle,$$ $\beta\triangleq  d(p,y)$. Then, for every \(v\in T_y\Omega\), set $\Lambda_0\triangleq\max\{\Lambda_+,\Lambda_-\}$ and
\begin{align*}
& \bigl|\operatorname{Hess}_y\psi(v,v)\bigr|
\le 
L_{\log}^{\pm}(\beta)\|\xi\|\,\|v\|^2, \\ \text{where} \ \ \  L_{\log}^{\pm}(\beta)
\triangleq & \ 
\alpha_+(\beta)^3 
\left[
\frac16 L_R \beta^2 b_-(\beta)^3
+
\frac56 c_n\Lambda_0 \beta\, b_-(\beta)^2 c_-(\beta)
\right].
\end{align*}
\end{lemma}

\begin{proof}
Let \(u\triangleq \log_p(y)\), so \(\beta=\|u\|=d(p,y)\le r_p\le \rho\). By Lemma~\ref{lem:prox_local_convexity},
\(\overline{B_{r_p}(p)}\) is a strongly convex normal ball. Hence \(A_u\triangleq (d\exp_p)_u\) is invertible,
the geodesic \(\eta(\tau)\triangleq \exp_p(\tau u)\), \(\tau\in[0,1]\), is the unique minimizing geodesic from \(p\) to \(y\),
and \(\eta([0,1])\subset \overline{B_{r_p}(p)}\subset \mathcal M_{\mathrm{ex}}\). Thus along \(\eta\) we  use
\(\|R\|\le c_n\Lambda_0\) and \(\|\nabla R\|\le L_R\).
Let \(\gamma:(-\varepsilon,\varepsilon)\to \Omega\) be the geodesic with \(\gamma(0)=y\) and \(\dot\gamma(0)=v\), and define
\(Y(t)\triangleq \log_p(\gamma(t))\). Since \(\exp_p(Y(t))=\gamma(t)\), differentiation gives $A_{Y(t)}Y'(t)=\dot\gamma(t)$. Differentiating covariantly in \(t\) and evaluating at \(t=0\) yields
\[
(\nabla_{Y'(0)}A)_uY'(0)+A_uY''(0)=0,
\]
because \(\gamma\) is geodesic. Since \(\psi(\gamma(t))=\langle \xi,Y(t)\rangle\),
\[
\operatorname{Hess}_y\psi(v,v)=\frac{d^2}{dt^2}\Big|_{t=0}\psi(\gamma(t))
=\langle \xi,Y''(0)\rangle,
\]
\[
\bigl|\operatorname{Hess}_y\psi(v,v)\bigr|
\le \|\xi\|\,\|A_u^{-1}\|\,\|(\nabla_{Y'(0)}A)_uY'(0)\|.
\]

We first show \(\|A_u^{-1}\|\le \alpha_{+}(\beta)\). For \(w\in T_p\Omega\), decompose
\(w=w^\parallel+w^\perp\), where \(w^\parallel\in \mathbb Ru\) and \(w^\perp\perp u\).
The radial part is preserved: \(\|A_u(w^\parallel)\|=\|w^\parallel\|\).
For the orthogonal part, let \(\widehat u=u/\beta\), \(\bar\eta(s)\triangleq \exp_p(s\widehat u)\), \(0\le s\le \beta\), and
$\widetilde J(s)\triangleq (d\exp_p)_{s\widehat u}\!\big(s\,{w^\perp}/{\beta}\big).
$
Then \(\widetilde J\) is an orthogonal Jacobi field along the unit-speed geodesic \(\bar\eta\) with
\(\widetilde J(0)=0\),\newline \(D_s\widetilde J(0)=w^\perp/\beta\), and \(\widetilde J(\beta)=A_u(w^\perp)\).
By the metric comparison theorem~\cite[Thm.~11.10]{lee2018introduction} with the upper sectional-curvature bound \(\Lambda_+\) gives
\[
\|A_u(w^\perp)\|=\|\widetilde J(\beta)\|
\ge ({\sin(\sqrt{\Lambda_+}\beta)}/{(\sqrt{\Lambda_+}\beta)})\,\|w^\perp\|
=\alpha_{+}(\beta)^{-1}\|w^\perp\|.
\]
By Gauss' lemma~\cite[Thm.~6.9]{lee2018introduction}, \(A_u(w^\parallel)\perp A_u(w^\perp)\), hence \(\|A_uw\|\ge \alpha_{+}(\beta)^{-1}\|w\|\),~so
\[
\|A_u^{-1}\|\le \alpha_{+}(\beta),
\qquad
\|Y'(0)\|=\|A_u^{-1}v\|\le \alpha_{+}(\beta)\|v\|.
\]

Next, we estimate \((\nabla_\zeta A)_u w\). Fix \(\zeta,w\in T_p\Omega\), set
$$s_\varepsilon \triangleq  u+\varepsilon\zeta, \ \ 
\eta_\varepsilon(\tau)\triangleq \exp_p(\tau s_\varepsilon), \ \ 
\text{and} \ \ J^\varepsilon(\tau)\triangleq (d\exp_p)_{\tau s_\varepsilon}(\tau w).$$
Let \(\eta=\eta_0\), \(U\triangleq \partial_\tau\eta_\varepsilon|_{\varepsilon=0}=\dot\eta\),
\(V\triangleq \partial_\varepsilon\eta_\varepsilon|_{\varepsilon=0}\),
\(J\triangleq  J^0\), and \(Z\triangleq  d_\varepsilon J^\varepsilon|_{\varepsilon=0}\).
Since \(J^\varepsilon(1)=A_{u+\varepsilon\zeta}w\), we have $Z(1)=(\nabla_\zeta A)_u w$.

Also \(Z(0)=0\), \(D_\tau Z(0)=0\), and differentiating the Jacobi equation for \(J^\varepsilon\) gives $$D_\tau^2 Z+R(Z,U)U=-F$$ where
\begin{align*}
F&=(\nabla_U R)(V,U)J+(\nabla_V R)(J,U)U +R(D_\tau V,U)J \\
&\quad +2R(V,U)D_\tau J+R(J,D_\tau V)U+R(J,U)D_\tau V.
\end{align*}

We now replace the nonnegative-curvature Jacobi estimates by the
corresponding hyperbolic comparison estimates.  Let \(W\) be any Jacobi
field along \(\eta\) with \(W(0)=0\) and \(D_\tau W(0)=a\).  Write the same
field in the unit-speed parameter \(s=\beta\tau\) as $\widehat W(s)\triangleq W(s/\beta)$. Then \(\widehat W(0)=0\) and $D_s\widehat W(0)={a}/{\beta}$. By Rauch comparison~\cite[Thm.~11.9]{lee2018introduction} under the lower sectional-curvature bound $\sec_\Omega \ge -\Lambda_-$, we have
\[
    \|\widehat W(s)\|
    \le
    s\, b_-(s)\left\|{a}/{\beta}\right\|,
    \qquad
    \|D_s\widehat W(s)\|
    \le
    c_-(s)\left\|{a}/{\beta}\right\|.
\]
Returning to \(s=\beta\tau\), and using the monotonicity of \(b_-\) and
\(c_-\), gives
\[
    \|W(\tau)\|
    \le
    \tau b_-(\beta\tau)\|a\|
    \le
    \tau b_-(\beta)\|a\|,
\]
\[
    \|D_\tau W(\tau)\|
    =
    \beta\|D_s\widehat W(\beta\tau)\|
    \le
    c_-(\beta\tau)\|a\|
    \le
    c_-(\beta)\|a\|.
\]
Applying these estimates to \(J\) and \(V\), and using \(\|U\|=\beta\), we obtain
\[
    \|J(\tau)\|\le \tau b_-(\beta)\|w\|,
    \ 
    \|D_\tau J(\tau)\|\le c_-(\beta)\|w\|,
    \
    \|V(\tau)\|\le \tau b_-(\beta)\|\xi\|,
    \
    \|D_\tau V(\tau)\|\le c_-(\beta)\|\xi\|.
\]
and therefore
$\|F(\tau)\|
    \le
    \left(
    2L_R\beta^2\tau^2 b_-(\beta)^2
    +
    5c_n\Lambda_0\beta\tau b_-(\beta)c_-(\beta)
    \right)
    \|\xi\|\,\|w\|$.

For \(s\in[0,1]\) and \(a\in T_{\eta(s)}\Omega\), let \(G_s^a\) be the Jacobi field on \([s,1]\) solving
\[
D_\tau^2 G_s^a+R(G_s^a,U)U=0,\qquad G_s^a(s)=0,\qquad D_\tau G_s^a(s)=a.
\]
By the previous estimate, \(\|G_s^a(1)\|\le (1-s)b_{-}(\beta(1-s))\|a\|\leq(1-s)b_{-}(\beta)\|a\|.\)\  Define $$\widetilde Z(\tau)\triangleq -\int_0^\tau G_s^{F(s)}(\tau)\,ds.$$ A direct differentiation under the integral sign shows that \(\widetilde Z\) satisfies the same inhomogeneous Jacobi equation and initial conditions as \(Z\), hence \(\widetilde Z=Z\). Thus,
\[
\|Z(1)\|
\le b_{-}(\beta)\int_0^1 (1-s)\|F(s)\|\,ds
\le
\left(\frac16L_R\beta^2b_{-}(\beta)+\frac56c_n\Lambda_0\beta b_-(\beta)^2c_-(\beta)\right)\|\zeta\|\,\|w\|.
\]
Since \(Z(1)=(\nabla_\zeta A)_u w\), we have proved
\[
\|(\nabla_\zeta A)_u w\|
\le
\left(\frac16L_R\beta^2b_{-}(\beta)+\frac56c_n\Lambda_0\beta b_-(\beta)^2c_-(\beta)\right)\|\zeta\|\,\|w\|.
\]

Finally choose \(\zeta=w=Y'(0)\). Then combined with \(\|Y'(0)\|\le \alpha_{+}(\beta)\|v\|\), 
\begin{align*}
    \|Y''(0)\|&\le
\|A_u^{-1}\|\,\|(\nabla_{Y'(0)}A)_uY'(0)\|\\
&\le
\alpha_{+}(\beta)
\left(\frac16L_R\beta^2b_{-}(\beta)+\frac56c_n\Lambda_0\beta b_-(\beta)^2c_-(\beta)\right)\|Y'(0)\|^2.\\
&\le \alpha_{+}(\beta)^3
\left(\frac16L_R\beta^2b_{-}(\beta)+\frac56c_n\Lambda_0\beta b_-(\beta)^2c_-(\beta)\right)\|v\|^2
\end{align*}
Thus $\bigl|\operatorname{Hess}_y\psi(v,v)\bigr|{\le}L_{\log}^{\pm}(\beta)\|\xi\|\,\|v\|^2$
\end{proof}

The next lemma provides a global lower bound on the Hessian of $ g$ and smoothness estimates for $g$ and $h$. This lower bound need not be positive, so positivity in the proximal model will be recovered from the local proximal term on $\overline{B_{r_k}(y_k)}$.

\begin{lemma}\label{lem:gh_properties}
Let Assumptions~\ref{assum:bounded_curv} and~\ref{assum:bound_y} hold.
For every \(z\in \mathcal M_r\), \(y\in \mathcal M_{\mathrm{ex}}\), and \(v\in T_y\Omega\), yields $\delta_{\mathrm{ex}}\|v\|^2
\le
\operatorname{Hess}_y\!\big( (1/2) d^2(z,y)\big)(v,v)
\le
\zeta_{\mathrm{ex}}\|v\|^2$,
and therefore\[2w_+(x)\delta_{\mathrm{ex}}\|v\|^2
\le
\operatorname{Hess}_y g(v,v)
\le
2w_+(x)\zeta_{\mathrm{ex}}\|v\|^2,
\]
\[
2w_-(x)\delta_{\mathrm{ex}}\|v\|^2
\le
\operatorname{Hess}_y h(v,v)
\le
2w_-(x)\zeta_{\mathrm{ex}}\|v\|^2.
\]
\end{lemma}

\begin{proof}
If \(y=z\), then $\operatorname{Hess}_y\!\big(\frac12 d^2(z,\cdot)\big)(v,v)=\|v\|^2$, so the claim is immediate. Assume henceforth that \(y\neq z\). Define \(\rho_z(y)\triangleq  d(z,y)\). Then
$\rho_z(y)\le d(z,c)+d(c,y)\le r+\rho_{\mathrm{ex}}$.
Moreover, \(r+\rho_{\mathrm{ex}}<\operatorname{inj}_{\Omega}(z)\), so
\(\rho_z\) is smooth at \(y\).
By Lemma~\ref{lem:sqdist_curv}
\[
    \delta_+(\rho_z(y))\|v\|_y^2
    \le
    \operatorname{Hess}_{y}\!\left(\frac12 d^2(z,y)\right)[v,v]
    \le
    \zeta_-(\rho_z(y))\|v\|_y^2 .
\]

Since \(\delta_+\) is decreasing, \(\zeta_-\) is increasing, and \(\rho_z(y)\le r+\rho_{\mathrm{ex}}\),
\[
    \delta_+(\rho_z(y))
    \ge
    \delta_+(r+\rho_{\mathrm{ex}})
    =
    \delta_{\mathrm{ex}},
\quad
    \zeta_-(\rho_z(y))
    \le
    \zeta_-(r+\rho_{\mathrm{ex}})
    =
    \zeta_{\mathrm{ex}} .
\]
Summing the inequalities termwise gives the bounds for \(g\) and \(h\).
\end{proof}

\begin{lemma}\label{lem:prox_local_convexity}
Let Assumptions~\ref{assum:bounded_curv}, \ref{assum:bound_ex}, and \ref{assum:rho} hold. Fix \(p\in \operatorname{int}(\mathcal M_{\mathrm{ex}})\) and let
\[
0<r_p<\min\{\operatorname{dist}(p,\partial\mathcal M_{\mathrm{ex}}),\,\rho\}.
\]
Then \(B_{r_p}(p)\) is a strongly convex normal ball. Moreover, \(y\mapsto \frac12 d^2(p,y)\) is \(\delta_{+}(r_p)\)-strongly geodesically convex on
\(\overline{B_{r_p}(p)}\).
\end{lemma}

\begin{proof}
Since \(p\in \mathcal M_{\mathrm{ex}}\), \(\operatorname{inj}_\Omega(p)\ge \iota_{\rho_{\mathrm{ex}},c}\), and \(r_p<\operatorname{dist}(p,\partial\mathcal M_{\mathrm{ex}})\), we have \(\overline{B_{r_p}(p)}\subset \mathcal M_{\mathrm{ex}}\), and \(B_{r_p}(p)\) is a strongly convex normal ball by \cite[Lem.~3.2.1]{wintraecken2015ambient}.

Now fix \(y\in \overline{B_{r_p}(p)}\). If \(y=p\), then $
\operatorname{Hess}_p\!\big( (1/2) d^2(p,\cdot)\big)(v,v)=\|v\|^2\ge \delta_{+}(r_p)\|v\|^2$,
since \(\delta_{+}(r_p)\le 1\). If \(y\neq p\), define \(\ell=d(p,y)\). Then \(\ell\le r_p<\pi/(2\sqrt{\Lambda_+})\), and the Hessian comparison theorem for the distance function
\cite[Thm.~11.7]{lee2018introduction} gives
\[
\operatorname{Hess}_y\!\big( (1/2) d^2(z,y)\big)(v,v)\ge \delta_{+}(\ell)\|v\|^2. 
\]
Since \(\delta_{+}\) is decreasing and \(\ell\le r_p\), $\delta_{+}(\ell)\ge \delta_{+}(r_p)$, hence $\operatorname{Hess}_y\!\big( (1/2) d^2(p,y)\big)(v,v)\ge \delta_{+}(r_p)\|v\|^2$. 

Finally, let \(\gamma:[0,1]\to \overline{B_{r_p}(p)}\) be any minimizing geodesic. Because
\(B_{r_p}(p)\) is strongly convex, \(\gamma([0,1])\subset B_{r_p}(p)\). Therefore
\[
\frac{d^2}{dt^2}\big( (1/2) d^2(p,\gamma(t))\big)
=
\operatorname{Hess}_{\gamma(t)}\!\big( (1/2) d^2(p,\cdot)\big)(\dot\gamma(t),\dot\gamma(t))
\ge
\delta_{+}(r_p)\|\dot\gamma(t)\|^2.
\]
The standard one-dimensional characterization of strong convexity along geodesics now yields
\(\delta_{+}(r_p)\)-strong geodesic convexity on \(\overline{B_{r_p}(p)}\).
\end{proof}

The next proposition deduces that the chosen value of $\tau$ ensures that $\Phi_k$ is strongly convex by combining the Hessian bounds obtained above for each term.

\begin{proposition}\label{prop:Phi_p_strongly_convex}
Let Assumptions~\ref{assum:bounded_curv}, \ref{assum:bound_y},
\ref{assum:bound_ex}, and \ref{assum:rho} hold. Fix $p\in \operatorname{int}(\mathcal M_{\mathrm{ex}})$, and let $$0{<}r_p{<}\min\{\operatorname{dist}(p,\partial\mathcal M_{\mathrm{ex}}),\,\rho\},$$ and $$\tau>
({L_{\log}^{\pm}(r_p)\|\grad h(p)\|{-}2w_+(x)\delta_{\mathrm{ex}}})/{\delta_{+}(r_p)},$$
$$L_{\log}^{\pm}(r_k)\triangleq \alpha_+(r_k)^3\left[\frac16 L_Rr_k^2 b_-(r_k)^3+\frac56 c_n\Lambda_0r_k b_-(r_k)^2c_-(r_k)\right], \ \ 0\leq t\leq\rho,$$
then $\Phi_p(y)$ is \(\mu(p,\tau)\)-strongly geodesically convex on \(\overline{B_{r_p}(p)}\), where
\[
\mu(p,\tau) \triangleq 
2w_+(x)\delta_{\mathrm{ex}}
+\tau\delta_{+}(r_p)
-L_{\log}^{\pm}(r_p)\|\grad h(p)\|>0.
\]
\end{proposition}

\begin{proof}
Let \(y\in \overline{B_{r_p}(p)}\) and \(v\in T_y\Omega\). Then
\[
\operatorname{Hess}_y\Phi_p(v,v)
=
\operatorname{Hess}_y g(v,v)
-\operatorname{Hess}_y\!\bigl(\langle \grad h(p),\log_p(y)\rangle\bigr)(v,v)
+\frac{\tau}{2}\operatorname{Hess}_y d^2(p,y)(v,v).
\]
Since \(\overline{B_{r_p}(p)}\subset \mathcal M_{\mathrm{ex}}\), Lemma~\ref{lem:gh_properties} gives $\operatorname{Hess}_y g(v,v)\ge 2w_+(x)\delta_{\mathrm{ex}}\|v\|^2$. 

Next, Lemma~\ref{lemma:bound_hess_lin} applied with \(\xi=\grad h(p)\) yields
\[
\bigl|\operatorname{Hess}_y\!\bigl(\langle \grad h(p),\log_p(y)\rangle\bigr)(v,v)\bigr|
\le
L_{\log}^{\pm}(\beta)\|\grad h(p)\|\,\|v\|^2,
\qquad \beta\triangleq  d(p,y).
\]
Because \(\beta\le r_p\le \rho\) and \(t\mapsto L_{\log}^{\pm}(t)\) is nondecreasing on \([0,\rho]\),
\[
\bigl|\operatorname{Hess}_y\!\bigl(\langle \grad h(p),\log_p(y)\rangle\bigr)(v,v)\bigr|
\le
L_{\log}^{\pm}(r_p)\|\grad h(p)\|\,\|v\|^2.
\]
Finally, Lemma~\ref{lem:prox_local_convexity} implies $
\frac{\tau}{2}\operatorname{Hess}_y d^2(p,y)(v,v)\ge \tau\delta_{+}(r_p)\|v\|^2$, and,
\[
\operatorname{Hess}_y\Phi_p(v,v)
\ge
\Bigl(2w_+(x)\delta_{\mathrm{ex}}
+\tau\delta(r_p)
-L_{\log}^{\pm}(r_p)\|\grad h(p)\|\Bigr)\|v\|^2
=
\mu(p,\tau)\|v\|^2.
\]
By the assumed lower bound on \(\tau\), \(\mu(p,\tau)>0\). Since
\(\overline{B_{r_p}(p)}\) is a strongly convex normal ball by
Lemma~\ref{lem:prox_local_convexity}, this lower Hessian bound implies that
\(\Phi_p\) is \(\mu(p,\tau)\)-strongly geodesically convex on \(\overline{B_{r_p}(p)}\).
\end{proof}

While Proposition~\ref{prop:Phi_p_strongly_convex} guarantees strong convexity of $\Phi_p$ on $\overline{B_{r_p}(p)}$, the next lemma adds a further condition ensuring that its unique minimizer lies in $B_{r_p}(p)\subset \mathcal{M}_{\mathrm{ex}}$.

\begin{lemma}\label{lem:invariance_prox_step}
Let Assumptions~\ref{assum:bounded_curv}--\ref{assum:rho} hold.
Fix \(p\in \operatorname{int}(\mathcal M_{\mathrm{ex}})\), let \(\theta\in(0,1)\), and 
\[
r_p{=}\min\Bigl\{\theta\,\operatorname{dist}(p,\partial\mathcal M_{\mathrm{ex}}),\ \rho\Bigr\},
\  \ \tau {>}
\frac{L_{\log}^{\pm}(r_p)\|\grad h(p)\|{-}2w_+(x)\delta_{\mathrm{ex}}}{\delta_{+}(r_p)}
+\frac{2\|\grad f(p)\|}{\delta_{+}(r_p)\,r_p}.
\]
Then \(\Phi_p\) admits a unique minimizer on \(B_{r_p}(p)\subset \mathcal M_{\mathrm{ex}}\).
\end{lemma}

\begin{proof}
By Lemma~\ref{lem:prox_local_convexity}, \(\overline{B_{r_p}(p)}\) is a strongly convex normal ball. By Proposition~\ref{prop:Phi_p_strongly_convex}, the above bound on \(\tau\) implies that \(\Phi_p\) is strongly geodesically convex on \(\overline{B_{r_p}(p)}\). Since \(\overline{B_{r_p}(p)}\) is compact, \(\Phi_p\) attains a unique minimizer there.

It remains to show that the minimizer cannot lie on \(\partial B_{r_p}(p)\). Let \(z\in \partial B_{r_p}(p)\), and let
\(\gamma:[0,r_p]\to \overline{B_{r_p}(p)}\) be the unique unit-speed minimizing geodesic from \(p\) to~\(z\). Set \(\varphi(t)\triangleq \Phi_p(\gamma(t))\). Strong geodesic convexity gives $\varphi''(t)\ge \mu(p,\tau)>0$, on $t\in[0,r_p]$. Integrating we get, $\varphi'(r_p)\ge \varphi'(0)+\mu(p,\tau)r_p$. Since \(\grad \Phi_p(p)=\grad f(p)\),
\[
\varphi'(0)=\langle \grad f(p),\dot\gamma(0)\rangle \ge -\|\grad f(p)\|.
\]
By the assumed lower bound on \(\tau\), one has $\mu(p,\tau)>{2\|\grad f(p)\|}/{r_p}$, and therefore
\[
\varphi'(r_p)>-\|\grad f(p)\|+2\|\grad f(p)\|=\|\grad f(p)\|>0.
\]
Thus moving slightly inward from \(z\) along \(\gamma\) decreases \(\Phi_p\), contradicting the minimality of \(z\).
Hence, the unique minimizer lies in \(B_{r_p}(p)\).
\end{proof}

\begin{lemma}\label{lem:uniform_tau_bound}
Let Assumptions~\ref{assum:bounded_curv}, \ref{assum:bound_y},
\ref{assum:bound_ex}, and \ref{assum:rho} hold. Fix \(\underline r>0\),
\(\theta\in(0,1)\), and \(\eta_0>0\). For each
\(y\in \operatorname{int}(\mathcal M_{\mathrm{ex}})\) satisfying $\operatorname{dist}(y,\partial\mathcal M_{\mathrm{ex}})\ge \underline r$, define \newline
$r(y)\triangleq \min\Bigl\{\theta\,\operatorname{dist}(y,\partial\mathcal M_{\mathrm{ex}}),\,\rho\Bigr\}$, $\delta(y)\triangleq \delta_{+}(r(y))$, and
\[
\tau(y){\triangleq }
\max\left\{
\frac{L_{\log}^{\pm}(r(y))\|\grad h(y)\|{-}2w_+(x)\delta_{\mathrm{ex}}}{\delta(y)}
+\frac{2\|\grad f(y)\|}{\delta(y)\,r(y)}
{+}\eta_0,\;
1{-}2w_-(x)\delta_{\mathrm{ex}},\;
1
\right\}.
\]
Then there exists a constant \(\bar\tau<\infty\) such that $\tau(y)\le \bar\tau$ for every such \(y\).
\end{lemma}

\begin{proof}
Since \(\mathcal M_{\mathrm{ex}}\) is compact and \(f,h\) are smooth on
\(\mathcal M_{\mathrm{ex}}\),
\[
G_f\triangleq\sup_{z\in\mathcal M_{\mathrm{ex}}}\|\grad f(z)\|<\infty,
\qquad
G_h\triangleq\sup_{z\in\mathcal M_{\mathrm{ex}}}\|\grad h(z)\|<\infty .
\]
Moreover, $r_\theta\triangleq\min\{\theta r,\rho\}\le r(y)\le \rho$.
Since \(\rho<\pi/(2\sqrt{\Lambda_+})\) and \(\delta_+\) is decreasing, $$\delta_y\triangleq\delta_+(r(y))\ge \delta_+(\rho)=:\underline{\delta}>0.$$ Also \(L_{\log}^{\pm}\) is continuous and nondecreasing on \([0,\rho]\), so
$L_{\log}^{\pm}(r(y))
\le L_{\log}^{\pm}(\rho)=:\overline L_{\log}$. Therefore,
\[
\frac{
L_{\log}^{\pm}(r(y))\|\grad h(y)\|-2w_+(x)\delta_{\mathrm{ex}}
}{\delta_y}
+
\frac{2\|\grad f(y)\|}{\delta_y r(y)}
+\eta_0
\le
\frac{\overline L_{\log}G_h-2w_+(x)\delta_{\mathrm{ex}}}
{\underline{\delta}}
+
\frac{2G_f}{\underline{\delta}r_\theta}
+\eta_0 .
\]
Hence
$
\tau(y)
\le
\max\left\{
\frac{\overline L_{\log}G_h-2w_+(x)\delta_{\mathrm{ex}}}
{\underline{\delta}}
+
\frac{2G_f}{\underline{\delta}r_\theta}
+\eta_0,\,
1-2w_-(x)\delta_{\mathrm{ex}},\,1
\right\}
=:\bar\tau<\infty$.
\end{proof}

\subsection{Proof of Theorem~\ref{thm:accumulation}}

With the strong convexity of the local model and the invariance of the proximal step established, each subproblem is well posed, each accepted step yields a quantitative decrease, and the summability of the step sizes implies stationarity of every accumulation point.
\begin{proof}
Set $\mathcal S_0\triangleq \{y\in \mathcal M_{\mathrm{ex}}:\ f(y,x)\le f(y_0,x)\}$.
Since \(f(\cdot,x)\) is continuous on \(\mathcal M_{\mathrm{ex}}\), the set \(\mathcal S_0\) is compact. Moreover, by the hypothesis $f(y_0,x)<\min_{y\in\partial\mathcal M_{\mathrm{ex}}} f(y,x)$, 
we have \(\mathcal S_0\cap \partial\mathcal M_{\mathrm{ex}}=\varnothing\), and therefore $\underline r\triangleq \operatorname{dist}(\mathcal S_0,\partial\mathcal M_{\mathrm{ex}})>0$.

Since \(r_k<\operatorname{dist}(y_k,\partial\mathcal M_{\mathrm{ex}})\), one has $\mathcal M_k\subset \mathcal M_{\mathrm{ex}}$. By Proposition~\ref{prop:Phi_p_strongly_convex}, each model function \(\Phi_k\) is strongly geodesically convex on \(\mathcal M_k\). Hence, in the exact case, the subproblem has a unique minimizer, and Lemma~\ref{lem:invariance_prox_step} implies that this minimizer belongs to \(B_{r_k}(y_k)\subset \mathcal M_{\mathrm{ex}}\). In the inexact case, the algorithm chooses \(y_{k+1}\in \mathcal M_k\subset \mathcal M_{\mathrm{ex}}\) by construction. Thus, in both cases, $y_{k+1}\in \mathcal M_{\mathrm{ex}}$. Now, let $d_k\triangleq  d(y_k,y_{k+1})$, and derive the descent estimates. In the exact method, \(\Phi_k(y_{k+1})\le \Phi_k(y_k)\), so
\[
g(y_k)-g(y_{k+1})
+\langle \grad h(y_k),\log_{y_k}(y_{k+1})\rangle
\ge ({\tau_k}/{2})d_k^2.
\]
Since \(y_k,y_{k+1}\in \mathcal M_k\), Lemma~\ref{lem:gh_properties}
applies along their minimizing geodesic,
\[
h(y_{k+1})-h(y_k)
\ge
\langle \grad h(y_k),\log_{y_k}(y_{k+1})\rangle
+w_-(x)\delta_{\mathrm{ex}}\,d_k^2.
\]
Adding the two inequalities yields $f(y_k,x)-f(y_{k+1},x)
\ge
\big({\tau_k}/{2}+w_-(x)\delta_{\mathrm{ex}}\big)d_k^2$.
Since \(\tau_k\ge 1-2w_-(x)\delta_{\mathrm{ex}}\), we obtain
\begin{equation}\label{eq:descent-exact-new}
f(y_{k+1},x)\le f(y_k,x)- (1/2) d_k^2.
\end{equation}

For the inexact method, geodesic convexity of \(\Phi_k\) on \(\mathcal M_k\) gives
\[
\Phi_k(y_k)-\Phi_k(y_{k+1})
\ge
\left\langle \grad \Phi_k(y_{k+1}),\log_{y_{k+1}}(y_k)\right\rangle.
\]
Hence $\Phi_k(y_k)-\Phi_k(y_{k+1})
\ge
-\|\grad \Phi_k(y_{k+1})\|\,d_k$. By the stopping rule,
\[
\|\grad \Phi_k(y_{k+1})\|\le \min(\varepsilon_k,\zeta d_k)\le \zeta d_k,
\ \text{so} \ 
\Phi_k(y_k)-\Phi_k(y_{k+1})\ge -\zeta d_k^2.
\]
Expanding \(\Phi_k\) and using again the lower second-order bound for \(h\) from
Lemma~\ref{lem:gh_properties},
\[
f(y_k,x)-f(y_{k+1},x)
\ge
\left({\tau_k}/{2}+w_-(x)\delta_{\mathrm{ex}}-\zeta\right)d_k^2.
\]
Since \(\tau_k\ge 1-2w_-(x)\delta_{\mathrm{ex}}\) and \(\zeta=1/4\), it follows that
\begin{equation}\label{eq:descent-inexact-new}
f(y_{k+1},x)\le f(y_k,x)- (1/4) d_k^2.
\end{equation}

In either case, \(f(y_{k+1},x)\le f(y_k,x)\). By induction, $y_k\in \mathcal S_0$, for $k\ge 0$. Therefore, $\operatorname{dist}(y_k,\partial\mathcal M_{\mathrm{ex}})\ge \underline r$, for $k\ge 0$. Lemma~\ref{lem:uniform_tau_bound} then yields $\sup_{k\ge 0}\tau_k<\infty$.

Summing \eqref{eq:descent-exact-new} or \eqref{eq:descent-inexact-new} from \(k=0\) to \(N\) proves \eqref{main:exact}
\begin{align*}
    \kappa\sum_{k=0}^N d_k^2
\le
f(y_0,x)-f(y_{N+1},x)
\le
f(y_0,x)-f_*,
\ 
\min_{0\le k\le N} d_k^2
\le
\frac{1}{N+1}\sum_{k=0}^N d_k^2
\le
\frac{f(y_0,x)-f_*}{\kappa(N+1)}
\end{align*}
In particular, $\sum_{k=0}^\infty d_k^2<\infty$, so $d_k\to 0$. Since \(\{y_k\}\subset \mathcal S_0\) and \(\mathcal S_0\) is compact, the sequence has at least one accumulation point \(\bar y\in \mathcal S_0\subset \mathcal M_{\mathrm{ex}}\). This proves (1). 

Next we show stationarity. Let \(\bar y\) be an accumulation point and take
\(y_{k_j}\to \bar y\). Since \(d_{k_j}\to0\), also \(y_{k_j+1}\to\bar y\).
Writing \((d\log_{y_k})^*_{y_{k+1}}:T_{y_k}\Omega\to T_{y_{k+1}}\Omega\)
for the adjoint of \(d(\log_{y_k})_{y_{k+1}}\), define the residual
\[
e_k\triangleq \grad \Phi_k(y_{k+1})
=
\grad g(y_{k+1})
-
(d\log_{y_k})^*_{y_{k+1}}\,\grad h(y_k)
-
\tau_k\log_{y_{k+1}}(y_k).
\]
In the exact case, \(e_k=0\). In the inexact case, $\|e_k\|=\|\grad\Phi_k(y_{k+1})\|\le \varepsilon_k\to 0$, since \(\sum_{k=0}^\infty \varepsilon_k<\infty\). Because \(\bar y\in \operatorname{int}(\mathcal M_{\mathrm{ex}})\) and \(y_{k_j},y_{k_j+1}\to \bar y\), for \(j\) large enough the points
\(y_{k_j}\) and \(y_{k_j+1}\) lie in a common normal neighborhood of \(\bar y\).
Let $P_j:T_{y_{k_j}}\Omega\to T_{\bar y}\Omega$,
$Q_j:T_{y_{k_j+1}}\Omega\to T_{\bar y}\Omega
$
be parallel transport along the corresponding minimizing geodesics, and define
$\mathcal A_j\triangleq Q_j\,(d\log_{y_{k_j}})^*_{y_{k_j+1}}\,P_j^{-1}\text{ on }T_{\bar y}\Omega$.
Applying \(Q_j\) to the identity defining \(e_{k_j}\), we obtain
\[
Q_j e_{k_j}
=
Q_j\grad g(y_{k_j+1})
-
\mathcal A_j\bigl(P_j\grad h(y_{k_j})\bigr)
-
\tau_{k_j}Q_j\log_{y_{k_j+1}}(y_{k_j}).
\]
Since \(g\) and \(h\) are smooth on \(\mathcal M_{\mathrm{ex}}\),
\[
Q_j\grad g(y_{k_j+1})\to \grad g(\bar y),
\qquad
P_j\grad h(y_{k_j})\to \grad h(\bar y).
\]
Moreover, on a common normal neighborhood of the diagonal, \((p,y)\mapsto (d\log_p)^*_y\) is smooth, so $\mathcal A_j\to \operatorname{Id}_{T_{\bar y}\Omega}$ in operator norm. Finally,
\[
\|\tau_{k_j}Q_j\log_{y_{k_j+1}}(y_{k_j})\|
=
\tau_{k_j}d(y_{k_j+1},y_{k_j})
\le
\Bigl(\sup_{k\ge 0}\tau_k\Bigr)d_{k_j}\to 0.
\]
Since \(\|Q_j e_{k_j}\|=\|e_{k_j}\|\to 0\), passing to the limit gives
\[
0=\grad g(\bar y)-\grad h(\bar y)=\grad f(\bar y,x).
\]
Thus every accumulation point is stationary, proving (2).
\end{proof}

\section{Improved rates on real-analytic manifolds}
\label{subsec:analytic_rates}

The previous section established descent of the objective and stationarity of all accumulation points of the proximal DC iterates. In this section, we show that these qualitative conclusions can be sharpened when the ambient manifold is real analytic. Indeed, on $\operatorname{int}(\mathcal M_{\mathrm{ex}})$ the objective $f(\cdot,x)$ is real analytic and therefore satisfies a Kurdyka--\L{}ojasiewicz inequality near its stationary points. Combined with the descent and relative-error estimates for the proximal scheme, this yields convergence of the whole sequence, with corresponding finite, linear, or sublinear rates determined by the KL exponent.

\begin{assumption}\label{ass:analytic}
The ambient Riemannian manifold $(\Omega,d)$ is real analytic.
\end{assumption}

\begin{corollary}\label{cor:kl}
Let Assumptions~\ref{assum:bounded_curv}, \ref{assum:bound_y},
\ref{assum:bound_ex}, and~\ref{ass:analytic} hold. Fix $x\in\mathbb R^q$.
Then the function $f(\cdot,x)$ in~\eqref{eq:main2} is real analytic on
$\operatorname{int}(\mathcal M_{\mathrm{ex}})$. Consequently, for every
$\bar y\in\operatorname{int}(\mathcal M_{\mathrm{ex}})$, the function $f(\cdot,x)$ satisfies the
Riemannian Kurdyka--\L{}ojasiewicz property at $\bar y$. More precisely, there exist a
neighborhood $U$ of $\bar y$, constants $c>0$, $\vartheta\in[0,1)$, and $\delta>0$ such that
\begin{equation}\label{eq:analytic_loja}
\|\grad f(y,x)\| \ge c\,|f(y,x)-f(\bar y,x)|^{\vartheta},
\qquad y\in U,\ |f(y,x)-f(\bar y,x)|<\delta.
\end{equation}
Equivalently, the KL inequality holds at $\bar y$ with $\varphi(s)=({1}/{c(1-\vartheta)})\,s^{1-\vartheta}$, for $s\in[0,\delta)$. 
\end{corollary}

\begin{proof}
Fix $i\in\{1,\dots,m\}$. Since $y_i\in \overline{B_r(c)}$ and
$y\in\mathcal M_{\mathrm{ex}}$, it is implied that
\[
d(y_i,y)\le r+\rho_{\mathrm{ex}}<\iota_{r,c}\le \operatorname{inj}_{\Omega}(y_i),
\]
the point $y$ lies in the injectivity ball of $y_i$. Because $(\Omega,d)$ is real analytic,
the map $\exp_{y_i}$ is real analytic on a neighborhood of $0\in T_{y_i}\Omega$, hence its
local inverse $\log_{y_i}$ is real analytic on $\operatorname{int}(\mathcal M_{\mathrm{ex}})$. Therefore $y\mapsto d^2(y_i,y)=\|\log_{y_i}(y)\|_{y_i}^2$ is real analytic on $\operatorname{int}(\mathcal M_{\mathrm{ex}})$. Since $f(\cdot,x)$ is a finite linear combination of these squared-distance terms, it is real analytic on $\operatorname{int}(\mathcal M_{\mathrm{ex}})$.

Now choose a real-analytic chart $\psi:U_0\to V_0\subset\mathbb R^n$ around $\bar y$, set $\bar z\triangleq\psi(\bar y)$, and define $\tilde f\triangleq f\circ \psi^{-1}$. Then $\tilde f$ is real analytic on $V_0$. By the classical Łojasiewicz gradient inequality in Euclidean space, there exists a neighborhood $V\subset V_0$ of $\bar z$, constants $c_0>0$, $\vartheta\in[0,1)$, and $\delta>0$ such that
\[
\|\nabla \tilde f(z)\| \ge c_0\,|\tilde f(z)-\tilde f(\bar z)|^{\vartheta},
\qquad z\in V,\ |\tilde f(z)-\tilde f(\bar z)|<\delta.
\]
Shrinking $U\triangleq\psi^{-1}(V)$ if necessary, the Euclidean norm of $\nabla \tilde f$ and the Riemannian norm of $\grad f$ are equivalent on $U$, so there exists $c>0$ such that \eqref{eq:analytic_loja} holds. This is exactly the KL property with the displayed power-type desingularizing function.
\end{proof}

\begin{lemma}\label{lem:relative_error_analytic}
Let Assumptions~\ref{assum:bounded_curv}--\ref{ass:analytic} hold, and let $d_k \triangleq d(y_k,y_{k+1})$. Then there exists a constant $C_{\mathrm{rel}}>0$ such that $
\|\grad f(y_{k+1},x)\|\le C_{\mathrm{rel}}\,d_k$ for all sufficiently large $k$. More precisely, one may take
\[
C_{\mathrm{rel}}=
\begin{cases}
L_A G_h+L_h+\bar\tau, & \text{exact step},\\
L_A G_h+L_h+\bar\tau+\zeta, & \text{inexact step},
\end{cases}
\]
where $G_h\triangleq\sup_{z\in\mathcal M_{\mathrm{ex}}}\|\grad h(z)\|$, $\bar\tau$ is the uniform bound from Lemma~\ref{lem:uniform_tau_bound}, and $L_A,L_h>0$ are the local Lipschitz constants defined as $L_A = L_{\log}^{\pm}(\rho), L_h \le 2w_-(x)\zeta_{\rm ex}$.
\end{lemma}

\begin{proof}
By the proof of Theorem~\ref{thm:accumulation}, all iterates belong to the compact strict sublevel set $$\mathcal S_0\triangleq\{y\in M_{\rm ex}: f(y,x)\le f(y_0,x)\}\subset\operatorname{int}(\mathcal M_{\mathrm{ex}}),$$ and Lemma~\ref{lem:uniform_tau_bound} gives $\tau_k\le \bar\tau$ for all $k$. Since $h$ is smooth on the compact set $\mathcal M_{\mathrm{ex}}$, we also have $G_h<\infty$.

Let $\mathcal{D}_\rho\triangleq\{(p,y)\in M_{\rm ex}\times M_{\rm ex}: d(p,y)\le \rho\}$. For \((p,y)\in \mathcal D_\rho\), define $$A(p,y)\triangleq(d_y\log_p)^*:T_p\Omega\to T_y\Omega.$$ Since the pairs in \(\mathcal D_\rho\) lie in a common normal neighborhood, \(A\) is
well defined and smooth on $\mathcal D_\rho$. For \(A(p,y)\triangleq(d_y\log_p)^*\) and \(\psi_\xi(z)\triangleq\langle \xi,\log_p(z)\rangle_p\), let \(\gamma\) be the unit-speed minimizing geodesic from \(p\) to \(y\). Then, for any unit parallel field \(V\) along \(\gamma\),
\[
\left\langle A(p,y)\xi-P_{p\to y}\xi,V_y\right\rangle_y
=
\int_0^{d(p,y)}
\operatorname{Hess}_{\gamma(s)}\psi_\xi(\dot\gamma(s),V(s))\,ds .
\]
Hence Lemma~\ref{lemma:bound_hess_lin} gives $\|A(p,y)\xi-P_{p\to y}\xi\|\le L_A\|\xi\|\,d(p,y)$ with $L_A\triangleq L_{\log}^{\pm}(\rho)$.
Also, by Lemma~\ref{lem:gh_properties},
\(h\) is \(L_h\)-smooth on \(M_{\rm ex}\), with $L_h\le 2w_-(x)\zeta_{\rm ex}$. Therefore,
\begin{align}
&\|A(p,y)\grad h(p)-\grad h(y)\| \notag\\
&\le
\|A(p,y)\grad h(p)-P_{p\to y}\grad h(p)\|
+\|P_{p\to y}\grad h(p)-\grad h(y)\| \notag\\
&\le (L_A G_h+L_h)d(p,y),
\qquad (p,y)\in\mathcal D_\rho .
\label{eq:A_grad_h_rewrite}
\end{align}
Now $y_{k+1}\in \mathcal M_k\subset \overline{B_\rho(y_k)}$, so $(y_k,y_{k+1})\in\mathcal D_\rho$.
For the exact step,
\[
0=\grad g(y_{k+1})-A(y_k,y_{k+1})\grad h(y_k)-\tau_k\log_{y_{k+1}}(y_k).
\]
Hence, using \eqref{eq:A_grad_h_rewrite},
\begin{align*}
\|\grad f(y_{k+1},x)\|
&=\|\grad g(y_{k+1})-\grad h(y_{k+1})\| \\
&\le \|A(y_k,y_{k+1})\grad h(y_k)-\grad h(y_{k+1})\|
+\tau_k\,\|\log_{y_{k+1}}(y_k)\| \\
&\le (L_A G_h+L_h+\bar\tau)\,d_k.
\end{align*}
For the inexact step, let $r_k\triangleq\grad\Phi_k(y_{k+1})$. 

Then
$\|r_k\|\le \min\{\varepsilon_k,\zeta d_k\}\le \zeta d_k$, and
$r_k=\grad g(y_{k+1})-A(y_k,y_{k+1})\grad h(y_k)-\tau_k\log_{y_{k+1}}(y_k).$
Therefore,
\begin{align*}
\|\grad f(y_{k+1},x)\|
&\le \|r_k\|+
\|A(y_k,y_{k+1})\grad h(y_k)-\grad h(y_{k+1})\|+
\tau_k\,\|\log_{y_{k+1}}(y_k)\| \\
&\le (L_A G_h+L_h+\bar\tau+\zeta)\,d_k.
\end{align*}
\end{proof}

\begin{theorem}\label{thm:analytic_rates_combined}
Let Assumptions~\ref{assum:bounded_curv}--\ref{ass:analytic} hold.
Let $\{y_k\}\subset \mathcal M_{\mathrm{ex}}$ be the sequence generated by
Algorithm~\ref{alg:rpdca}, and define
\[
S\triangleq\Bigl\{\bar y\in\mathcal M_{\mathrm{ex}}:\ \exists\,k_j\to\infty\ \text{with}\ y_{k_j}\to \bar y\Bigr\},
\qquad d_k\triangleq d(y_k,y_{k+1}).
\]
Then $S=\{y_\star\}$ for some stationary point $y_\star$ of $f(\cdot,x)$, and therefore
$y_k\to y_\star$. Moreover, if $\vartheta\in[0,1)$ is a KL exponent of $f(\cdot,x)$ at $y_\star$, then:
\begin{itemize}[leftmargin=2em,itemsep=1pt,topsep=1pt]
    \item if $\vartheta=0$, the sequence converges in finite time;
    \item if $\vartheta\in(0,1/2]$, then  $d(y_k,y_\star)=\mathcal O(\rho^k)$ for some $\rho\in(0,1)$;
    \item if $\vartheta\in(1/2,1)$, then $    d(y_k,y_\star)=\mathcal O\big(k^{-\frac{1-\vartheta}{2\vartheta-1}}\big)$.
\end{itemize}
\end{theorem}

\begin{proof}
Set $f_k\triangleq f(y_k,x)$. By the proof of  Theorem~\ref{thm:accumulation}, for every $k\ge0$
\begin{equation}\label{eq:descent_analytic_rewrite}
f_k-f_{k+1}\ge \kappa\,d_k^2.
\end{equation}
Hence $\{f_k\}$ is nonincreasing. Since $f$ is continuous on the compact
set $\mathcal M_{\mathrm{ex}}$, it is bounded below there, and thus $f_k\downarrow \ell$ for some
$\ell\in\mathbb R$. Again by Theorem~\ref{thm:accumulation}, every cluster point of $\{y_k\}$ is stationary for
$f(\cdot,x)$, and $\{y_k\}\subset \mathcal S_0\subset\operatorname{int}(\mathcal M_{\mathrm{ex}})$, where
$\mathcal S_0$ is the compact strict sublevel set introduced in the proof of Theorem~\ref{thm:accumulation}. The cluster set $S$ is nonempty and compact.
If $\bar y\in S$, there exists a subsequence $y_{k_j}\to\bar y$, hence by continuity
$f(\bar y,x)=\lim_j f_{k_j}=\ell$. Therefore $f(\cdot,x)$ is constant on $S$.

By Corollary~\ref{cor:kl}, $f(\cdot,x)$ satisfies the KL property at every point of $S$. Since $S$ is compact and $f \to \ell$ on $S$, the standard uniformized KL lemma yields $\varepsilon>0$, $\eta>0$, and a concave function $\varphi:[0,\eta)\to[0,\infty)$ such that
\begin{equation}\label{eq:uniform_KL_rewrite}
\varphi'(f(y,x)-\ell)\,\|\grad f(y,x)\|\ge 1
\end{equation}
whenever $\operatorname{dist}(y,S)<\varepsilon$ and $\ell<f(y,x)<\ell+\eta$.
Because $\{y_k\}$ is contained in the compact set $\mathcal S_0$, one has
$\operatorname{dist}(y_k,S)\to0$; otherwise a subsequence staying a fixed positive distance
from $S$ would admit a further convergent subsequence with limit in~$S$, a contradiction.
Since also $f_k\to \ell$, estimate~\eqref{eq:uniform_KL_rewrite} holds at $y_k$ for all large enough~$k$.

Now reindex Lemma~\ref{lem:relative_error_analytic} to obtain a constant $C_{\mathrm{rel}}>0$ such
that
\begin{equation}\label{eq:relative_error_rewrite}
\|\grad f(y_k,x)\|\le C_{\mathrm{rel}}\,d_{k-1}
\end{equation}
for all sufficiently large $k$. By concavity of $\varphi$,
\[
\varphi(f_k-\ell)-\varphi(f_{k+1}-\ell)
\ge \varphi'(f_k-\ell)(f_k-f_{k+1}).
\]
Combining this with \eqref{eq:uniform_KL_rewrite},
\eqref{eq:descent_analytic_rewrite}, and \eqref{eq:relative_error_rewrite}, we obtain, for all
large enough $k$,
\[
\varphi(f_k-\ell)-\varphi(f_{k+1}-\ell)
\ge \frac{f_k-f_{k+1}}{\|\grad f(y_k,x)\|}
\ge \frac{\kappa}{C_{\mathrm{rel}}}\,\frac{d_k^2}{d_{k-1}}.
\]
Hence $d_k\le (1/2) d_{k-1} +({C_{\mathrm{rel}}}/{2\kappa})\bigl(\varphi(f_k-\ell)-\varphi(f_{k+1}-\ell)\bigr)$. Summing this inequality from $k=k_0+1$ to $N$, where $k_0$ is large enough for all previous estimates to hold, gives
\[
\frac12\sum_{k=k_0+1}^{N-1} d_k + d_N
\le \frac12 d_{k_0}
+\frac{C_{\mathrm{rel}}}{2\kappa}\,\varphi(f_{k_0+1}-\ell).
\]
Therefore $\sum_{k=0}^\infty d_k<\infty$. The sequence has finite length, so it is Cauchy and,
since $\mathcal M_{\mathrm{ex}}$ is compact, there exists $y_\star\in\mathcal M_{\mathrm{ex}}$ such that
$y_k\to y_\star$. Hence $S=\{y_\star\}$. As every cluster point is stationary by
Theorem~\ref{thm:accumulation}, we also have $\grad f(y_\star,x)=0$.

It remains to derive the rates. By Corollary~\ref{cor:kl}, there exist a neighborhood $U$ of
$y_\star$, constants $c>0$, $\vartheta\in[0,1)$, and $\delta>0$ such that
\[
\|\grad f(y,x)\|\ge c\,|f(y,x)-f(y_\star,x)|^{\vartheta},
\qquad y\in U,\ |f(y,x)-f(y_\star,x)|<\delta.
\]
Choose a real-analytic chart $\psi:U\to V\subset\mathbb R^n$ around $y_\star$, set
$z_k\triangleq\psi(y_k)$ and $\tilde f\triangleq f\circ\psi^{-1}$. After shrinking $U$ if necessary, the chart and
its inverse are bi-Lipschitz on $U$, and the Euclidean and Riemannian gradient norms are
uniformly equivalent there. Consequently, for all large enough $k$,
\[
\tilde f(z_k)-\tilde f(z_{k+1})\ge a\,\|z_{k+1}-z_k\|^2,
\qquad
\|\nabla \tilde f(z_k)\|\le b\,\|z_k-z_{k-1}\|
\]
for some constants $a,b>0$. Thus the charted sequence $\{z_k\}$ satisfies the standard
Euclidean sufficient-decrease and relative-error conditions, and the classical KL rate theorem
(see, e.g.,~\cite[Theorem~2]{attouch2009convergence}) applies to $\tilde f$ at
$z_\star\triangleq\psi(y_\star)$. Therefore:
\begin{itemize}[leftmargin=2em,itemsep=1pt,topsep=1pt]
    \item if $\vartheta=0$, then $z_k$ (and hence $y_k$) is eventually constant;
    \item if $\vartheta\in(0,1/2]$, then $z_k$ converges $R$-linearly to $z_\star$;
    \item if $\vartheta\in(1/2,1)$, then $\|z_k-z_\star\|=\mathcal O\big(k^{-\frac{1-\vartheta}{2\vartheta-1}}\big)$.
\end{itemize}
Since the chart is bi-Lipschitz, the same rate statements hold for $d(y_k,y_\star)$.
\end{proof}

\section{Numerical Analysis}\label{sec:numerics}

In this section, we provide numerical examples to illustrate the efficiency of our approach. In addition to global FRIDA, we also test local FRIDA, as defined in Remark~\ref{rem:global_local_frida}.

\subsection{Regression on the Sphere}
We consider Fr\'echet regression with predictor \(X\in\mathbb R\) and response \(Y\in S^2\), where $S^2=\{x\in\mathbb{R}^3:\|x\|=1\}$, with tangent space at $p$ defined as $T_pS^2=\{v\in\mathbb{R}^3:\langle v,p\rangle=0\}$, and geodesic distance $d(x,y)=\arccos(\langle x,y\rangle)$.

\subsubsection{Regression on Geodesic Data} 

For illustration, we first consider a simple case with three responses on a common geodesic segment of \(S^2\), with predictors \(\{0,0.5,1\}\) and responses
\(\{y_0,y_{0.5},y_1\}\). We choose an extrapolating test predictor
\(x_{\rm test}=1.87\), which lies outside the sufficient predictor region where
our theory guarantees interiority of minimizers. Nevertheless, the objective is
well defined on the chosen normal ball, and a minimizer is observed
numerically in this example. Thus, this experiment stress-tests
Algorithm~\ref{alg:rpdca} beyond the conservative sufficient safe-region
condition.
\begin{figure}[ht]
    \centering
    \includegraphics[width=\linewidth]{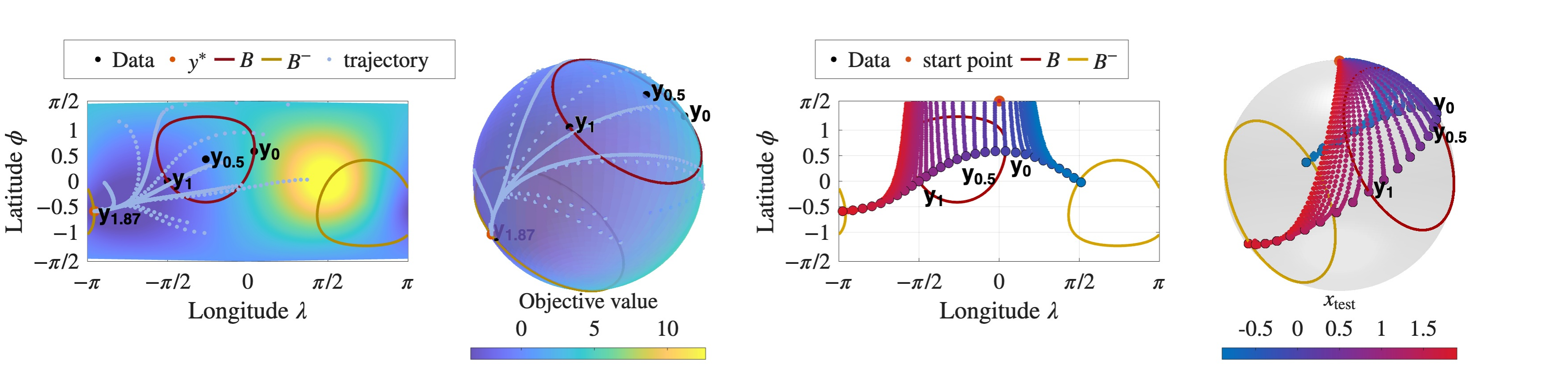}
    \makebox[\textwidth][c]{%
        \makebox[0.23\textwidth][c]{\text{(a)}}%
        \makebox[0.23\textwidth][c]{\text{(b)}}%
        \makebox[0.23\textwidth][c]{\text{(c)}}%
        \makebox[0.23\textwidth][c]{\text{(d)}}%
    }
     
    \caption{Global Fr\'echet Regression on data lying on a geodesic of \(S^2\). 
    Panels (a)--(b) show the weighted Fr\'echet objective \(f(\cdot,x)\) in longitude--latitude coordinates and on the sphere. Black points are the data \(\{y_0,y_{0.5},y_1\}\), the orange point is a stationary point, the light blue curves are trajectories from random initializations, and the red and gold curves denote the boundaries of \(B\) and \(B^{-}\). 
    Panels (c)--(d) show the geodesic regression with various \(x_{\mathrm{test}}\). The red dot is the start point. The color of the trajectories is proportional to  \(x_{\mathrm{test}}\). }
    \label{fig:dc_sphere_geodesic}
\end{figure} Figure~\ref{fig:dc_sphere_geodesic}(a)--(b) shows the weighted Fr\'echet objective on \(S^2\). Panel (a) gives the longitude--latitude projection, and Panel (b) shows the corresponding visualization on the sphere. The black data points correspond to \(\{y_0,y_{0.5},y_1\}\), and the orange point marks a stationary point. The light blue curves are trajectories generated by Algorithm~\ref{alg:rpdca} from random initializations. The red and gold curves denote the boundaries of the geodesic ball \(B\subset \mathcal M_r\) and its antipodal counterpart \(B^{-}\), respectively. The color map represents the functional value \(f(\cdot,x)\). In both views, the trajectories from different initial points converge to the same stationary point, indicating the stability of the proposed method in the selected region.

Additionally, under the same setup, we test Algorithm~\ref{alg:rpdca} when finding the response $y$ corresponding to different predictors $ x$ varying over $[-0.9,1.9]$ with a fixed initialization point. 

Figure~\ref{fig:dc_sphere_geodesic}(c)--(d) shows the geodesic regression
experiment with several test predictor values \(x_{\mathrm{test}}\). Panel (c) gives the longitude--latitude projection, and Panel (d) shows the corresponding visualization on the sphere. The red dot marks the start point, and the color of each trajectory corresponds to the value of \(x_{\mathrm{test}}\). In both views, the estimated points converge to stationary points of the corresponding weighted Fr\'echet objectives, showing that the method consistently identifies the regression estimates across different test predictor values.

For the next illustration, we add noise under the same geodesic ground-truth configuration as in the previous experiments. The predictors are sampled at 20 equally spaced values in [0,1], and the corresponding responses are the 20 points along the geodesic at those values. To generate noisy observations, for each response point \(y_i\in S^2\), we sample a Gaussian vector \(z_i\in\mathbb{R}^3\), project it onto the tangent space \(T_{y_i}S^2\), scale the projected vector by a noise level parameter \(\sigma=0.1\), and map it back to the sphere using the exponential map. The resulting observations are $
Y_i=\exp_{y_i}(v_i)$, and $v_i=\sigma\bigl(z_i-\langle z_i,y_i\rangle y_i\bigr)$, where \(z_i\sim\mathcal{N}(0,I_3)\). Figure~\ref{fig:dc_sphere_noise} (a)--(b) shows the noisy geodesic regression experiment on \(S^2\). Panel (a) gives the longitude--latitude projection, and Panel (b) shows the corresponding visualization on the sphere. The true geodesic curve is plotted in blue, the DCA regression estimate is shown in orange, and the noisy manifold-valued observations are displayed in green. The initialization point is marked by a purple star. In both views, the estimated curve closely follows the underlying geodesic despite the intrinsic noise in the observations, indicating that the regression procedure recovers the main geodesic structure of the data.
\subsubsection{Regression on Spiral Data}
We next consider a spherical regression example with a ground-truth curve
\[
m(x)=\Bigl(\sqrt{1-x^2}\cos(\pi x),\ \sqrt{1-x^2}\sin(\pi x),\ x\Bigr),
\qquad x\in(0,1),
\]
which forms a spiral-like path on \(S^2\). Observations are generated by adding tangent-space noise at \(m(X_i)\) and mapping back to the sphere using the exponential map. We then compare local and global Fr\'echet regression fits. 

Figure~\ref{fig:dc_sphere_noise}(c)--(d) shows the spiral-noise experiment on \(S^2\). Panel (c) gives the longitude--latitude projection, and Panel (d) maps the same curves onto the sphere. The true spiral response is shown in blue, the local Fr\'echet estimate in orange, and the global Fr\'echet estimate in purple. Noisy observations are plotted in green, and the initialization is marked by a star. Compared with the global estimator, the local estimator follows the spiral more closely and captures its local variation, while the global estimator recovers the overall trend but smooths out part of the local geometric
structure.
\begin{figure}[ht]
    \centering
    \includegraphics[width=\linewidth]{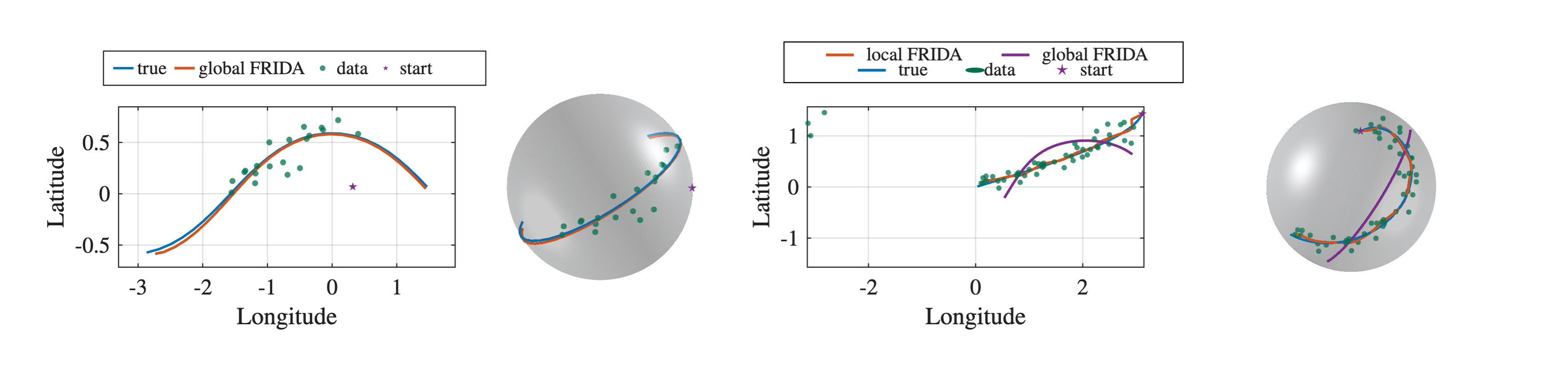}
    \makebox[\textwidth][c]{%
        \makebox[0.23\textwidth][c]{\text{(a)}}%
        \makebox[0.23\textwidth][c]{\text{(b)}}%
        \makebox[0.23\textwidth][c]{\text{(c)}}%
        \makebox[0.23\textwidth][c]{\text{(d)}}%
    }
     
    \caption{
    Fr\'echet regression on noisy data on \(S^2\).
    Panels (a)--(b) show global Fr\'echet regression for noisy observations lying near a geodesic, in longitude--latitude coordinates and on the 3D sphere, respectively. Panels (c)--(d) compare local and global Fr\'echet regression for spiral-noise data in the same two views. Blue denotes the true curve, green the noisy observations, orange the FRIDA local estimate, and purple the global estimate or initialization point. The estimates recover the main spherical regression structure despite the presence of intrinsic noise.
    }
    \label{fig:dc_sphere_noise}
\end{figure}

\subsection{Comparison with GD on $S^2 \times S^1$}
We generate synthetic responses on the product manifold \(\mathcal M=S^2\times S^1\), writing each response as
\(y=(p,\theta)\), with \(p\in S^2\subset\mathbb R^3\) and \(\theta\in[0,2\pi)\). We use the product metric $g_{\mathcal M}=g_{S^2}^{\mathrm{round}}\oplus g_{S^1}^{\mathrm{flat}}$, where the two factors carry the standard round and angular metrics.

We take \(n=40\) equally spaced predictors \(x_i\in[0,1]\). The noiseless regression function starts from \(y_{\mathrm{base}}=((0,0,1),0)\). 

Its spherical component moves from the north pole in the tangent direction \(e=(1,0,0)\in T_{(0,0,1)}S^2\), and its circular component evolves on \(S^1\). Using $\alpha(x)=1.40(3x^2-2x^3)$, which smooths the spherical motion near the endpoints, we define $m(x)
    =
    \left(
        \big(\sin\alpha(x),0,\cos\alpha(x)\big),
        0.80\pi x \ \mathrm{mod}\ 2\pi
    \right)$, for $x\in[0,1]$.
Since \(\max_x\alpha(x)=1.40<\pi/2\), the noiseless spherical component remains
in the open hemisphere centered at \((0,0,1)\).

Intrinsic noise is added independently to each factor. On \(S^2\), a Gaussian vector in \(\mathbb R^3\) is projected onto
\(T_{p_{\mathrm{true}}(x_i)}S^2\), normalized, scaled by a Gaussian amplitude with standard deviation \(\sigma_{S^2}=0.045\), and mapped back by the exponential map. On \(S^1\), Gaussian angular noise with standard deviation \(\sigma_{S^1}=0.035\) is added. This gives noisy responses \(y_i=(p_i,\theta_i)\in S^2\times S^1\). The product manifold satisfies \(0\leq K\leq 1\), and along the chosen regression curve, the effective curvature is bounded above by approximately \(0.41\).

For the optimization comparison, we keep only the test predictors \(x_{\mathrm{test}}\) whose global Fr\'echet regression weights contain at least one negative value. For each such \(x_{\mathrm{test}}\), GD and FRIDA solve the weighted Fr\'echet problem from the same noisy response. GD uses at most \(500\) iterations, while FRIDA uses at most \(500\) outer iterations and \(1000\) inner iterations per subproblem, with all gradient tolerances set to \(10^{-8}\). We report the final objective value, outer iteration counts, FRIDA inner iteration counts, and the final and best gradient norms. 

Figure~\ref{fig:s2s1_gd_dca_comparison}(a) summarizes the GD--FRIDA comparison over test predictors \(x_{\mathrm{test}}\) whose global weights include negative values. The four panels report the best Riemannian gradient norm \(\min_k\|\operatorname{grad}F(y_k)\|\), final objective value, outer iteration count, and FRIDA inner iteration count. FRIDA typically achieves lower gradient norms and requires far fewer outer iterations than GD, while both methods achieve nearly identical final objective values. The FRIDA inner counts remain well below the prescribed limit. Figure~\ref{fig:s2s1_gd_dca_comparison}(b) shows representative convergence trajectories. For readability, the gradient-norm Panel shows only the first 100 GD iterations, together with the full FRIDA trajectory. FRIDA reaches a small gradient norm much faster than GD. The objective-value Panel shows that both methods decrease the objective and converge to the same final value, with the displayed final values agreeing with the shown precision.

\begin{figure}[t]
    \centering

    \begin{subfigure}[t]{0.98\textwidth}
        \centering
        \includegraphics[width=\textwidth]{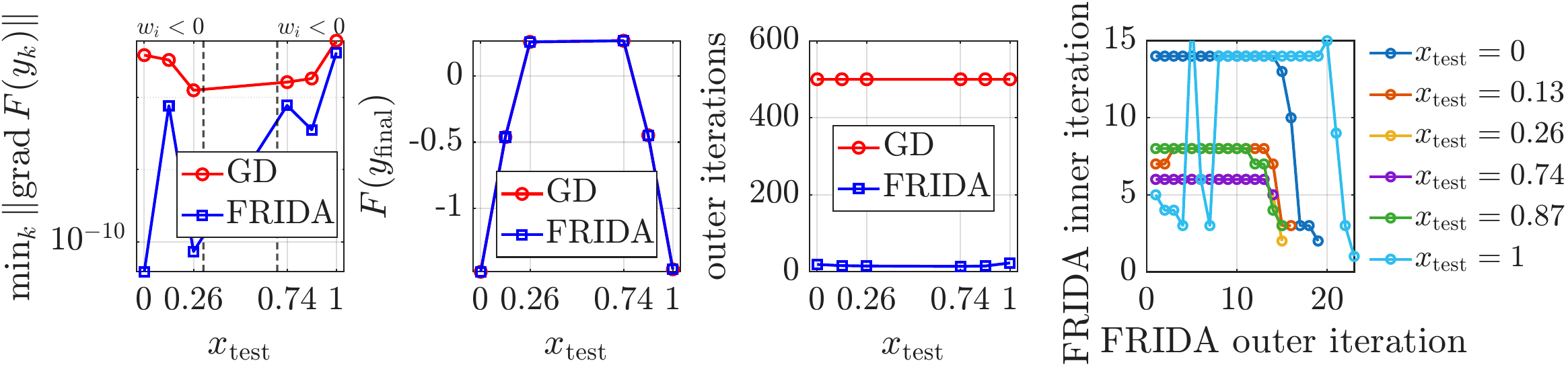}
       
        \caption{
        Summary of selected test predictors with negative global weights.
        From left to right, the Panels report the best Riemannian gradient norm,
        final objective value, number of outer iterations, and FRIDA inner iteration
        counts.       
        }
        \label{fig:s2s1_gd_dca_summary}
    \end{subfigure}

    \vspace{0.6em}

    \begin{subfigure}[t]{0.98\textwidth}
        \centering
        \includegraphics[width=\textwidth]{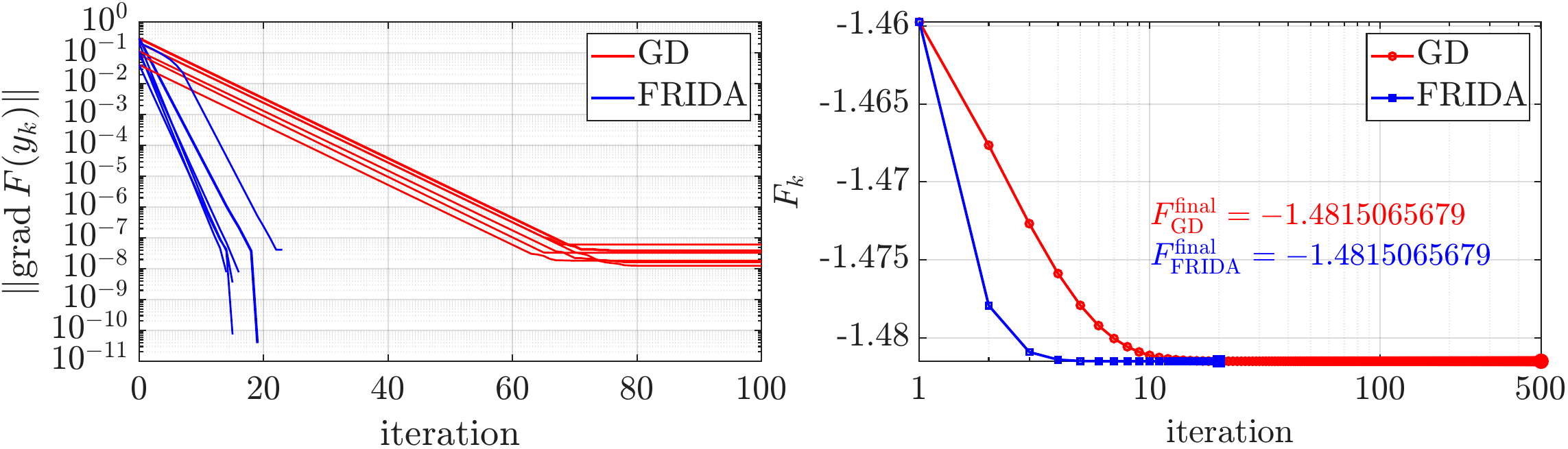}
       
        \caption{
        Representative convergence trajectory.
        The left Panel shows the first 100 GD gradient-norm iterations together
        with the full FRIDA gradient-norm trajectory. The right Panel shows the
        objective-value histories and the final objective values reached by both
        methods.       
        }
        \label{fig:s2s1_gd_dca_trajectory}
    \end{subfigure}

    \caption{
    Comparison between Riemannian gradient descent and FRIDA on \(S^2\times S^1\). The top row summarizes performance over selected test predictors \(x_{\mathrm{test}}\) whose global regression weights include negative values. FRIDA attains smaller best gradient norms and requires substantially fewer outer iterations than GD, while both methods obtain nearly identical final objective values. The FRIDA inner iteration counts remain well below the prescribed inner iteration budget, indicating a stable solution of the local surrogate subproblems. The bottom row shows a representative convergence trajectory: FRIDA achieves a lower gradient norm much faster than GD, and both methods converge to the same final objective value.
    }
    \label{fig:s2s1_gd_dca_comparison}
\end{figure}

\begin{remark}[Why FRIDA can outperform gradient descent] FRIDA is not uniformly better than gradient descent for all nonconvex problems, but it is well suited to the weighted Fr\'echet objective considered here. While GD uses only first-order information from the full objective \(F\) and may need small steps in ill-conditioned signed-distance landscapes, FRIDA exploits the decomposition \(F=g-h\). By linearizing \(-h\) and minimizing a locally convex surrogate involving \(g\), each FRIDA update can make more
structured progress than a single gradient step.
\end{remark}

\subsection{Spiral Regression on the Torus}

In the final experiment, we consider regression on an embedded torus with
angular coordinates \((\theta,\phi)\in S^1\times S^1\), where \(\theta\)
parametrizes the central circle and \(\phi\) the cross-sectional circle. For
major radius \(R\) and minor radius \(r\), we use $F(\theta,\phi)=
\bigl((R+r\cos\phi)\cos\theta,\,
(R+r\cos\phi)\sin\theta,\,
r\sin\phi\bigr)$,
whose induced metric is $ds^2=(R+r\cos\phi)^2d\theta^2+r^2d\phi^2$.
Thus, motion in the \(\theta\)-direction is scaled by \(R+r\cos\phi\), while
motion in the \(\phi\)-direction is scaled by \(r\).

The sectional curvature is $K(\phi)=\frac{\cos\phi}{r(R+r\cos\phi)}$. Hence, the outer region has positive curvature, the inner region has negative curvature, and the transition regions near \(\phi=\pi/2\) and \(\phi=3\pi/2\) have curvature close to zero.

Unlike the sphere experiment, the torus experiment uses approximate local
geometry, since exact intrinsic operations on the embedded torus generally require geodesic boundary-value solves \cite{jantzen2010torusGeodesics}. We
work in a small angular patch inside the normal/convexity regime, where squared
distances are smooth and locally convex, metric variation is mild, and wrapping
or cut-locus effects are avoided. We approximate
\[
d^2\bigl((\theta_1,\phi_1),(\theta_2,\phi_2)\bigr)
\approx
(R+r\cos\phi_{\rm mid})^2(\theta_1-\theta_2)^2
+
r^2(\phi_1-\phi_2)^2,
\]
where \(\phi_{\rm mid}\) is the midpoint angular coordinate, and use the
associated local orthonormal frame for logarithm/exponential maps. This matches
the metric quadratic approximation to squared geodesic distance up to
higher-order curvature terms
\cite[Sec.~4]{viaclovsky_riemannian_geometry_2011}. Thus, the torus experiment
is a robustness study under approximate local geometry, not an exact intrinsic
Fr\'echet regression experiment.

\subsubsection{Local torus with global weights}

We first study global Fr\'echet regression on a local patch of the embedded
torus. The data are generated in angular coordinates \((\theta,\phi)\). We set $\theta_0=0, \phi_0=\frac{\pi}{2}, R=2.0, r=0.7$. At the patch center, the \(\theta\)-metric scale is $A_0=R+r\cos\phi_0=R$, which converts angular displacement in \(\theta\) into local arclength. For \(x_i\in[0,1]\), the noiseless curve is chosen as a straight line in approximate orthonormal coordinates, with
\(L_{\mathrm{total}}=1.45\) and \(\alpha=\pi/3\):
\[
    \theta(x_i)
    =
    \theta_0+
    \frac{L_{\mathrm{total}}\cos\alpha}{A_0}
    \left(x_i-\frac12\right),
    \qquad
    \phi(x_i)
    =
    \phi_0+
    \frac{L_{\mathrm{total}}\sin\alpha}{r}
    \left(x_i-\frac12\right).
\]
The responses are obtained by mapping \((\theta(x_i),\phi(x_i))\) to the
embedded torus in \(\mathbb R^3\). 

This curve stays in a controlled local patch
while crossing both positive and negative curvature regions; along it, $-0.767 \le K(\phi) \le 0.397$. Noisy observations are generated by
\[
    \theta_i^{\mathrm{obs}}
    =
    \theta(x_i)+\frac{\sigma_{\mathrm{intrinsic}}}{A_0}\varepsilon_i^\theta,
    \qquad
    \phi_i^{\mathrm{obs}}
    =
    \phi(x_i)+\frac{\sigma_{\mathrm{intrinsic}}}{r}\varepsilon_i^\phi,
    \qquad
    \varepsilon_i^\theta,\varepsilon_i^\phi\sim N(0,1),
\]
with \(\sigma_{\mathrm{intrinsic}}=0.04\). This scaling makes the noise approximately isotropic under the local metric \[ds^2\approx A_0^2d\theta^2+r^2d\phi^2.\] The curve remains in the local normal/convexity regime described above, and the regression weights are the global affine weights. 

Fig.~\ref{fig:torus_regression_combined}(a)--(b) shows global Fr\'echet regression on a local patch of the embedded torus.  Panel~(a) gives the angular-coordinate view \((\theta,\phi)\), where black points denote noisy observations, the blue curve is the true response curve, the red curve is the global regression estimate, and the purple marker is the optimization start point. Panel~(b) maps the same objects onto the embedded torus in \(\mathbb R^3\). The estimate closely matches the true curve within the local patch, indicating that the method accurately recovers the trajectory in this controlled setting.
\begin{figure}[t]
    \centering
    \includegraphics[width=\textwidth]{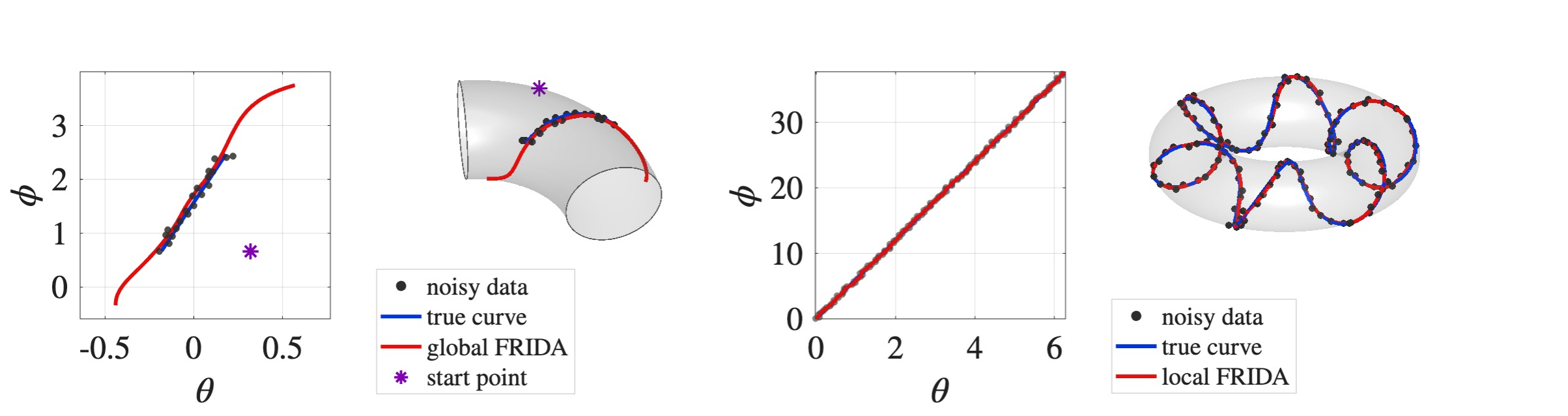}

    \vspace{-0.5em}
    \makebox[\textwidth][c]{%
        \makebox[0.23\textwidth][c]{\text{(a)}}%
        \makebox[0.23\textwidth][c]{\text{(b)}}%
        \makebox[0.23\textwidth][c]{\text{(c)}}%
        \makebox[0.23\textwidth][c]{\text{(d)}}%
    }
\vspace{-0.5cm}
    \caption{
    Fr\'echet regression experiments on the embedded torus.
    Panels (a)--(b) show global Fr\'echet regression on a controlled local patch:
    (a) gives the angular-coordinate view and
    (b) maps the same objects onto the local embedded torus patch in \(\mathbb R^3\).
    Panels (c)--(d) show local Fr\'echet regression along a closed curve covering
    the full torus:
    (c) gives the unwrapped angular-coordinate view and
    (d) shows the corresponding embedded curve in \(\mathbb R^3\).
    In all panels, black points denote noisy observations, blue curves denote
    the true trajectory, and red curves denote the DCA-based regression
    estimate. Purple stars mark the initialization points in panels (a)--(b). In both settings, the regression estimate closely follows
    the true trajectory from noisy observations.   
    }
    \label{fig:torus_regression_combined}
\end{figure}

\subsubsection{Global torus experiment with local weights}

We next consider regression on the full embedded torus, where the true response curve wraps around the entire surface. We set \(R=2.0\), \(r=0.7\), \(x\in[0,6]\), and generate the noiseless curve in angular coordinates with \(\theta_0=0\), \(\phi_0=0\), \(N_\theta=1\), and \(N_\phi=6\):
\[
    \theta(x)=\theta_0+2\pi N_\theta
    \frac{x-x_{\min}}{x_{\max}-x_{\min}},
    \qquad
    \phi(x)=\phi_0+2\pi N_\phi
    \frac{x-x_{\min}}{x_{\max}-x_{\min}}.
\]
Thus, the curve winds once in the central direction and six times in the
cross-sectional direction. Since it covers the full cross-sectional circle, it
passes through the full curvature range of the embedded torus. For
\(R=2.0\) and \(r=0.7\),
$-1.099 \le K \le 0.529$,
so the trajectory repeatedly visits positive, zero, and negative curvature
regions.

We add Gaussian noise to the unwrapped angles and then wrap modulo \(2\pi\):
\[
    \theta_i^{\mathrm{obs}}
    =
    \theta(x_i)+0.035\,\varepsilon_i^\theta,
    \qquad
    \phi_i^{\mathrm{obs}}
    =
    \phi(x_i)+0.035\,\varepsilon_i^\phi,
    \qquad
    \varepsilon_i^\theta,\varepsilon_i^\phi\sim N(0,1).
\]

Although the trajectory is global, each regression is computed locally. For a
test point \(x_0\), we use observations in a local predictor window of
half-width
\[
    \Delta x_{\mathrm{loc}}(x_0)
    =
    0.85\,\rho_{\mathrm{safe}}/
    \max\{\operatorname{speed}(x_0),10^{-12}\},
    \qquad
    \rho_{\mathrm{safe}}=0.55,
\]
with \(\Delta x_{\mathrm{loc}}\in[0.04,0.35]\), and set
\(h_{\mathrm{loc}}=0.45\Delta x_{\mathrm{loc}}\). Local observations are chosen using periodic distance in \(x\), with at least
8 observations per subproblem, and normalized Gaussian weights. This keeps each
local regression within an approximate intrinsic radius controlled by
\(\rho_{\mathrm{safe}}\), so that each subproblem remains in the local
normal/convexity regime described above, even though the full curve covers the
entire torus. 

Fig.~\ref{fig:torus_regression_combined}(c)--(d) shows local Fr\'echet
regression along a closed curve on the full embedded torus. Panel~(c) gives the
unwrapped angular-coordinate view, where the curve appears nearly linear because
it winds once in \(\theta\) and six times in \(\phi\). Black points denote
noisy observations, the blue curve is the true trajectory, and the red curve is
the local DCA estimate.
Panel~(d) maps the same curves onto the embedded torus in~\(\mathbb R^3\).
The close agreement between the orange and blue curves shows that the proposed
method recovers the trajectory despite angular wrapping and noise.

\section{Conclusions} \label{sec:conclusion}
We studied signed Fr\'echet regression on complete Riemannian manifolds with two-sided sectional-curvature bounds and developed FRIDA, a proximal DC framework for its computation. Since the regression weights may have mixed signs, the objective is an affine combination of squared distances and is generally nonconvex, with possible nonsmoothness near cut loci. By working on strongly convex normal balls with an adaptive proximal term, we showed that the local subproblems are well posed and strongly geodesically convex. Our analysis separates the two curvature bounds: the upper bound controls convexity radii and lower Hessian bounds, while the lower bound controls Jacobi-field growth and the smoothness constants for the logarithm linearization. Under explicit signed-weight conditions, we proved the existence and interiority of minimizers, descent of the exact and inexact FRIDA iterations, stationarity of all accumulation points, and the complexity bound \(O((N+1)^{-1/2})\). In the real-analytic case, the Kurdyka--{\L}ojasiewicz framework further gives full-sequence convergence with rates determined by the KL exponent. These results provide a theoretical foundation for FRIDA-type methods for signed Fr\'echet regression under local two-sided curvature control. Future work includes sharpening local rates, developing statistically consistent inexact solvers, and extending the framework to broader metric or stratified spaces where only local comparison estimates are available.

\bibliographystyle{plain}
\bibliography{Ref}

@article{hansen1982large,
  title={Large sample properties of generalized method of moments estimators},
  author={Hansen, Lars Peter},
  journal={Econometrica: Journal of the econometric society},
  pages={1029--1054},
  year={1982},
  publisher={JSTOR}
}

@book{lee2018introduction,
  title={Introduction to {R}iemannian manifolds},
  author={Lee, John M},
  volume={2},
  year={2018},
  publisher={Springer}
}

@article{lin2024type,
  title={A Type of Nonlinear {F}r\'echet Regressions},
  author={Lin, Lu and Chen, Ze},
  journal={arXiv preprint arXiv:2403.17481},
  year={2024}
}

@article{bennett2023variational,
  title={The variational method of moments},
  author={Bennett, Andrew and Kallus, Nathan},
  journal={Journal of the Royal Statistical Society Series B: Statistical Methodology},
  volume={85},
  number={3},
  pages={810--841},
  year={2023},
  publisher={Oxford University Press US}
}

@misc{viaclovsky_pcmi_curvature,
  author       = {Jeff A. Viaclovsky},
  title        = {Critical Metrics for {R}iemannian Curvature Functionals},
  note         = {Expanded version of IAS/PCMI lecture notes, 2013},
  year         = {2013},
  howpublished = {\url{https://www.math.uci.edu/~jviaclov/lecturenotes/PCMI_Lectures_Final.pdf}}
}

@misc{viaclovsky_riemannian_geometry_2011,
  author       = {Jeff A. Viaclovsky},
  title        = {Math 865: Topics in {R}iemannian Geometry},
  year         = {2011},
  note         = {Course notes},
  howpublished = {\url{https://www.math.uci.edu/~jviaclov/courses/865_Fall_2011.pdf}}
}

@article{afsari2013convergence,
  title   = {On the Convergence of Gradient Descent for Finding the {R}iemannian Center of Mass},
  author  = {Afsari, Bijan and Tron, Roberto and Vidal, Ren{\'e}},
  journal = {SIAM Journal on Control and Optimization},
  volume  = {51},
  number  = {3},
  pages   = {2230--2260},
  year    = {2013},
  publisher = {SIAM}
}

@article{takatsu2011wasserstein,
  title   = {Wasserstein Geometry of {G}aussian Measures},
  author  = {Takatsu, Asuka},
  journal = {Osaka Journal of Mathematics},
  volume  = {48},
  number  = {4},
  pages   = {1005--1026},
  year    = {2011}
}

@inproceedings{altschuler2021averaging,
  title     = {Averaging on the {Bures--Wasserstein} Manifold: Dimension-Free Convergence of Gradient Descent},
  author    = {Altschuler, Jason M. and Chewi, Sinho and Gerber, Patrik and Stromme, Austin J.},
  booktitle = {Advances in Neural Information Processing Systems},
  volume    = {34},
  pages     = {22132--22145},
  year      = {2021}
}

@article{martinez2024convergence,
  title={Convergence and trade-offs in {R}iemannian gradient descent  and {R}iemannian proximal point},
  author={Mart{\'\i}nez-Rubio, David and Roux, Christophe and Pokutta, Sebastian},
  journal={arXiv preprint arXiv:2403.10429},
  year={2024}
}

@book{do1992riemannian,
author = {Carmo, Manfredo Perdigao do.},
address = {Boston},
booktitle = {{R}iemannian geometry},
isbn = {0817634908},
language = {eng},
lccn = {91037377},
publisher = {Birkhäuser},
series = {Mathematics. Theory and applications},
title = {{R}iemannian geometry / Manfredo do Carmo ; translated by Francis Flaherty.},
year = {1992},
}

@unpublished{hosseini2015convergence,
author = {{Seyedehsomayeh Hosseini}},
title = {Convergence of nonsmooth descent methods via {K}urdyka–{L}ojasiewicz inequality on {R}iemannian manifolds},
publisher = {Institut für Numerische Simulation (INS)},
year = 2015,
month = nov,
volume = 1523
}

@article{kurdyka2000proof,
  title={Proof of the gradient conjecture of {R}. {T}hom},
  author={Kurdyka, Krzysztof and Mostowski, Tadeusz and Parusi{\'n}ski, Adam},
  journal={Annals of Mathematics},
  pages={763--792},
  year={2000},
  publisher={JSTOR}
}

@article{attouch2009convergence,
  title={On the convergence of the proximal algorithm for nonsmooth functions involving analytic features},
  author={Attouch, Hedy and Bolte, J{\'e}r{\^o}me},
  journal={Mathematical Programming},
  volume={116},
  number={1},
  pages={5--16},
  year={2009},
  publisher={Springer}
}

@article{petersen2019frechet,
  title={{F}r{\'e}chet regression for random objects with Euclidean predictors},
  author={Petersen, Alexander and M{\"u}ller, Hans-Georg},
  journal={The Annals of Statistics},
  volume={47},
  number={2},
  pages={691--719},
  year={2019},
  publisher={JSTOR}
}

@article{zhou2022network,
  title={Network regression with graph {L}aplacians},
  author={Zhou, Yidong and M{\"u}ller, Hans-Georg},
  journal={Journal of Machine Learning Research},
  volume={23},
  number={320},
  pages={1--41},
  year={2022}
}

@article{nava2024ridge,
  title={Ridge Regression for Manifold-valued Time-Series with Application to Meteorological Forecast},
  author={Nava-Yazdani, Esfandiar},
  journal={arXiv preprint arXiv:2411.18339},
  year={2024}
}

@article{im2025local,
  title={Local {F}r{\'e}chet regression with spherical predictors},
  author={Im, Chang Jun and Jeon, Jeong Min and Park, Byeong U},
  journal={Electronic Journal of Statistics},
  volume={19},
  number={2},
  pages={5313--5367},
  year={2025},
  publisher={The Institute of Mathematical Statistics and the Bernoulli Society}
}

@article{torres2022multivariate,
  title={Multivariate manifold-valued curve regression in time},
  author={Torres-Signes, A and Fr{\'\i}as, MP and Ruiz-Medina, MD},
  journal={arXiv preprint arXiv:2208.12585},
  year={2022}
}

@article{lin2021total,
  title={Total variation regularized {F}r{\'e}chet regression for metric-space valued data},
  author={Lin, Zhenhua and M{\"u}ller, Hans-Georg},
  journal={The Annals of Statistics},
  volume={49},
  number={6},
  pages={3510--3533},
  year={2021},
  publisher={Institute of Mathematical Statistics}
}

@article{horst1999dc,
  title={{DC} programming: overview},
  author={Horst, Reiner and Thoai, Nguyen V},
  journal={Journal of Optimization Theory and Applications},
  volume={103},
  pages={1--43},
  year={1999},
  publisher={Springer}
}

@inproceedings{tuy1984global,
  title={Global minimization of a difference of two convex functions},
  author={Tuy, Hoang},
  booktitle={Selected Topics in Operations Research and Mathematical Economics: Proceedings of the 8th Symposium on Operations Research, Held at the University of Karlsruhe, West Germany August 22--25, 1983},
  pages={98--118},
  year={1984},
  organization={Springer}
}

@inproceedings{hiriart1985generalized,
  title={Generalized differentiability/duality and optimization for problems dealing with differences of convex functions},
  author={Hiriart-Urruty, J-B},
  booktitle={Convexity and Duality in Optimization: Proceedings of the Symposium on Convexity and Duality in Optimization Held at the University of Groningen, The Netherlands June 22, 1984},
  pages={37--70},
  year={1985},
  organization={Springer}
}

@article{le2018dc,
  title={{DC} programming and {DCA}: thirty years of developments},
  author={Le Thi, Hoai An and Pham Dinh, Tao},
  journal={Mathematical Programming},
  volume={169},
  number={1},
  pages={5--68},
  year={2018},
  publisher={Springer}
}

@inproceedings{faust2023bregman,
  title={A {B}regman divergence view on the difference-of-convex algorithm},
  author={Faust, Oisin and Fawzi, Hamza and Saunderson, James},
  booktitle={International Conference on Artificial Intelligence and Statistics},
  pages={3427--3439},
  year={2023},
  organization={PMLR}
}

@article{yao2023globally,
  title={A globally convergent difference-of-convex algorithmic framework and application to log-determinant optimization problems},
  author={Yao, Chaorui and Jiang, Xin},
  journal={arXiv preprint arXiv:2306.02001},
  year={2023}
}

@article{weber2022class,
  title={On a class of geodesically convex optimization problems solved via {E}uclidean MM methods},
  author={Weber, Melanie and Sra, Suvrit},
  journal={arXiv preprint arXiv:2206.11426},
  year={2022}
}

@article{bergmann2024difference,
  title={The difference of convex algorithm on {H}adamard manifolds},
  author={Bergmann, Ronny and Ferreira, Orizon P and Santos, Elianderson M and Souza, Jo{\~a}o Carlos O},
  journal={Journal of Optimization Theory and Applications},
  volume={201},
  number={1},
  pages={221--251},
  year={2024},
  publisher={Springer}
}

@article{wintraecken2015ambient,
  title={Ambient and intrinsic triangulations and topological methods in cosmology},
  author={Wintraecken, Mathijs},
  year={2015}
}

@incollection{ziller2014riemannian,
  title={{R}iemannian manifolds with positive sectional curvature},
  author={Ziller, Wolfgang},
  booktitle={Geometry of Manifolds with Non-negative Sectional Curvature: Editors: Rafael Herrera, Luis Hern{\'a}ndez-Lamoneda},
  pages={1--19},
  year={2014},
  publisher={Springer}
}

@misc{jantzen2010torusGeodesics,
  title={Geodesics on the Torus and Other Surfaces of Revolution Clarified Using Undergraduate Physics Tricks with Bonus: Nonrelativistic and Relativistic Kepler Problems},
  author={Jantzen, Robert T.},
  year={2010},
  note={Lecture notes}
}

@book {MR1383587,
    AUTHOR = {Fan, J. and Gijbels, I.},
     TITLE = {Local polynomial modelling and its applications},
    SERIES = {Monographs on Statistics and Applied Probability},
    VOLUME = {66},
 PUBLISHER = {Chapman \& Hall, London},
      YEAR = {1996},
     PAGES = {xvi+341},
      ISBN = {0-412-98321-4},
   MRCLASS = {62G05 (62J02)},
  MRNUMBER = {1383587},
MRREVIEWER = {Theo\ Gasser},
}

@book{Lee_2013,
  title={Introduction to Smooth Manifolds},
  author={Lee, J.},
  isbn={9781441999825},
  lccn={2012945172},
  series={Graduate Texts in Mathematics},
  year={2012},
  publisher={Springer New York}
}

\end{document}